\documentclass[11pt]{amsart}
\newtheorem{Lemma}{Lemma}
\newtheorem{Theorem}{Theorem}
\newtheorem{Corollary}{Corollary}
\newtheorem{Definition}{Definition}

\def\Ack{\medskip\noindent {\bf Acknowledgements:}\ \ignorespaces}

\def\signra{\bigskip \begin{center}
{\sc Ricardo Alonso\par\vspace{3mm}
Department of Mathematics, University of Texas at Austin \\
Austin, TX 78712, U.S.A.
\par\vspace{3mm}
e-mail:} \tt{ralonso@math.utexas.edu}  \end{center}}

\def\signig{\bigskip \begin{center}
{\sc Irene M. Gamba\par\vspace{3mm}
Department of Mathematics and ICES, University of Texas at Austin \\
Austin, TX 78712, U.S.A.
\par\vspace{3mm}
e-mail:} \tt{gamba@math.utexas.edu}  \end{center}}

\begin{document}
\title[$L^{1}-L^{\infty}$ Maxwellian weighted bounds for
derivatives of any order]{Propagation of $L^{1}$ and $L^{\infty}$
Maxwellian weighted bounds for derivatives of solutions to the
homogeneous elastic Boltzmann Equation}
\author{Ricardo J. Alonso and Irene M. Gamba}
%
%
%
%
\begin{abstract}
We consider the $n$-dimensional space homogeneous Boltzmann equation
for elastic collisions for variable hard potentials with Grad
(angular) cutoff. We prove sharp moment inequalities, the
propagation of $L^1$-Maxwellian weighted estimates, and
consequently, the propagation $L^\infty$-Maxwellian weighted
estimates to all derivatives of the initial value problem associated
to the afore mentioned problem.

More specifically, we extend to all derivatives of the initial value
problem associated to this class of Boltzmann equations
corresponding  sharp moment (Povzner) inequalities and  time
propagation of $L^1$-Maxwellian weighted estimates  as originally
developed Bobylev \cite{Bobylev1} in the case of hard spheres in $3$
dimensions; an improved  sharp moments inequalities to a larger
class of angular cross sections and  $L^1$-exponential bounds  in
the case of stationary states to Boltzmann equations for inelastic
interaction problems with `heating' sources, by
Bobylev-Gamba-Panferov \cite{Bobylev}, where high energy tail decay
rates depend on the inelasticity coefficient and the the type of
`heating' source; and more recently, extended to variable hard
potentials with angular cutoff by Gamba-Panferov-Villani
\cite{Gamba} in the elastic case collision case and so
$L^1$-Maxwellian weighted estimated were shown to propagate if
initial states have such property. In addition, we also extend to
all derivatives the  propagation of $L^\infty$-Maxwellian weighted
estimates, proven in  \cite{Gamba}, to solutions of the initial
value problem to the Boltzmann equations for elastic collisions for
variable hard potentials with Grad (angular) cutoff.
\end{abstract}
\maketitle
%

\section{Introduction}

The study of propagation of $L^1$-Maxwellian weighted estimates to
solutions of the initial value for the $n$-dimensional space
homogeneous Boltzmann equation for elastic collisions for variable
hard potentials with Grad (angular) cutoff
 entices the study of summability properties of a corresponding
 series of the solution moments to all orders. This problem was
 addressed for the first time by Bobylev in \cite{Bobylev1} in the case of $3$
 dimension for the hard sphere problem, i.e. for constant angular cross section in
 the collision kernel.

 Previously,  the behavior of time propagating properties for the
moments of the solution to the initial value problem for the elastic
Boltzmann transport equation, in the space homogeneous regime for
hard spheres and variable hard potentials and  integrable angular
cross sections (Grad cutoff assumption) had been extensively
studied, but not their summability properties. In fact, the study of
Povzner estimates and propagation of moments of the solution to the
of variable hard potentials with Grad cutoff assumption, was
progressively understood in the work of Desvillettes
\cite{Desvillettes} and Wennberg \cite{Wennberg}, where the Povzner
estimates, a crucial tool for the moments control in the case of
variable hard spheres with the Grad cutoff assumption, where based
on pointwise estimates on the difference between pre and
post-collisional velocities of convex, isotropic weights functions
of the velocity in oreder to control their integral on the $n-1$
dimensional sphere, and consequently, and not sharp enough to obtain
summability of moments.

A  significant leap was developed by Bobylev \cite{Bobylev1} where
the first proof of summability properties of moments was
established, in the case of hard spheres in $3$ dimensions, showing
that $L^1$-Maxwellian weighted estimates propagates if the initial
data is in within that class. Among several new crucial techniques
that were developed in that fundamental paper, the is a significant
improvement of the Pozvner estimates based on the averaging
(integrals) on the $n-1$ dimensional sphere of convex, isotropic
weights functions of the velocity, for the case of variable hard
potentials with the Grad cutoff assumption. As a consequence it is
possible to established that, in the case of three dimensions
velocity, for hard spheres, the moments of the gain operator will
decay proportional to the order of the moment with respect to the
loss term uniformly in time, by means of  infinity  evolution
inequalities in terms of moments.  That key estimate yields
summability of a moments series, uniformly in time. Later, Bobylev,
Gamba and Panferov  \cite{Bobylev}, establish the sharpest version
of the Povzner inequality for elastic or  inelastic collisions,
using the approach of \cite{Bobylev1}, by a somehow reduced
argument, under the conditions that both the convex, isotropic
weights functions of the velocity and the angular part of the
angular cross section are non-decreasing. The two main ideas are to
pass to the center of mass – relative velocity variables and to use
the angular integration in to obtain more precise constants in the
corresponding inequalities. The summability property, which in the
work of Bobylev  \cite{Bobylev}, was done only for hard spheres in
three dimensions whose the angular cross section is constant,  was
extended, in \cite{Bobylev1}, to the case of bounded angular
section. In addition, the problem of stationary states to Boltzmann
equations for inelastic interaction problems with `heating' sources,
such as random heat bath, shear flow or self-similar transformed
problems, was addressed in \cite{Bobylev}, where $L^1$-exponential
bounds with decay rates depending of the inelasticity coefficient
and the the type of `heating' source were shown as well. In these
cases the authors showed $L^1$-exponential weighted decay bounds
with slower decay than  Maxwellian (i.e. $s<2$).

 In an interesting application
of these moments summability formulas, estimates and techniques,
Mouhot \cite{mohout06}, was able to establish (for the elastic case)
a result that proves the instantaneous `generation', of
$L^1$-exponential bounds uniformly in time, with only
$L^1_2\cap\L^2$ initial data, where the exponential rate is half of
the variable hard sphere exponent, under the same assumptions on the
angular function as in \cite{Bobylev}.

However, still for the elastic case, of variable hard potentials and
Grad cutoff assumption, neither \cite{Bobylev1} nor \cite{Bobylev},
addressed the propagation of $L^1$-Maxwellian weighted bounds,
uniform in time to solutions of the corresponding initial value
problem in $n$-dimensions with more realistic intramolecular
potentials. In a recent manuscript by Gamba, Panferov and Villani
\cite{Gamba}, showed  the $L^1$-Maxwellian weighted propagation
estimates and the provided a proof to the open problem of
propagation of $L^\infty$-Maxwellian weighted bounds, uniformly in
time. The Grad cutoff assumption was still assumed (integrability of
the angular part of the collision kernel) without the boundedness
condition, but a growth rate assumption on the angular singularity,
depending only on the velocity space dimension, that still keeps
integrability. The propagation of $L^\infty$-Maxwellian weighted
bounds combines the classical Carleman representation of the gain
operator with the $L^1$-Maxwellian bounds.

More specifically, the behavior for large velocities is commonly
called "high energy tails".  Under precise conditions, described in
\cite{Bobylev1} and \cite{Gamba}), it is known that for a solution
this asymptotic behavior is comparable in some way to
$\exp(-r|\xi|^{s})$ with $r,s$ positive numbers.  In   the case of
elastic interactions, it is known that $s=2$, provided the initial
state also has that behavior, i.e. it decays as a Maxwellian. This
is a revealing fact that says that a solution of the elastic initial
value problem for the $n$-dimensional Boltzmann equation, with
variable hard potential kernels  and singular integrable angular
cross section, decays like a Maxwellian for all times as long as the
initial state does it
as well.

\bigskip
In this present manuscript, we extend the results of \cite{Gamba} to
show both propagation of of $L^1$-Maxwellian and
$L^\infty$-Maxwellian weighted estimates to all derivatives of the
solution to the initial value problem to the space homogeneous
Boltzmann equations for elastic collisions for variable hard
potentials with an integrable angular singularity condition as in
\cite{Gamba}.

We first note that sharp Povzner inequalities
(\cite{Bobylev1}-\cite{Bobylev}-\cite{Gamba}) are, indeed,  the main
tool for the study of the solution's moments for the variable hard
potential models.  They control the decay of the moments of the gain
collision operator with respect to the moments of the loss collision
operator. This technique yields a control of the time derivative of
any higher order moment using the lower order ones. In particular,
one uses  the Boltzmann equation, in order  to build an infinite
system of sharp Povzner inequalities for each moments which can be
used, by arguing inductively, to control each moment uniformly in
time. As a result one obtains $L^1$-Maxwellian weighted estimates
and the corresponding $L^\infty$-Maxwellian weighted estimates in
the elastic interaction models in $n$-dimensions and  for variable
hard potential collision kernels with an integrable angular
singularity condition depending on the dimension $n$.

  Here we show that these results extend to the study of propagation of
$L^1$-Maxwellian  and $L^\infty$-Maxwellian weighted estimates to
any high order  derivatives of the solution to the $n$-dimensional
elastic Boltzmann equation for variable hard potentials.  In
particular, these bounds imply that if the initial derivatives of
the solution are controlled pointwise by the derivatives of a
Maxwellian
then this control propagates for all times.

The paper is organized as follows. After this introduction, section
2 presents the problem and the main Theorem~\ref{T3}. Section 3 focus
in finding sharp Povzner inequalities for the solution's derivatives.
All the computations regarding the derivatives of the collision operator
and the action of the differential collision operator on test functions
are presented in Lemmas~\ref{L1}, \ref{L2} and \ref{L3}.  Lemmas~\ref{L4},
\ref{L5}, \ref{L6} and \ref{L7} are
devoted to provide a suitable expression ready to use for the construction
of the mentioned system of inequalities on the derivative's moments.
In Lemma~\ref{L8} such a system of inequalities is presented.

Then in section 4, all previous results used to obtain information
for the solution's derivatives in the elastic case. Theorem~\ref{T1}
proves the control of moment's growth, and Theorem~\ref{T2} uses
Lemma~\ref{L8} to obtain a global in time bound for the derivative's
moments in the elastic case yielding the $L^1$-Maxwellian bounds to
derivatives of any order. Finally in section 5, we show that uniform
bounds on the moments of these derivatives lead to a pointwise
estimate. This is possible after using an $L^{\infty}-L^{1}$
Maxwellian weighted control on the gain collision operator as shown
in  \cite{Gamba} (see Theorems~\ref{T5} and \ref{T6} in the Appendix
A).  The Boltzmann equation and this control are sufficient to find
a time uniform pointwise control by Maxwellians for the solution's
derivatives of any order. Calculations of this fact are performed in
Theorem~\ref{T3}, where the $L^{\infty}$-Maxwellian bound is shown.
\section{Preliminaries and Main Result}
This section presents the assumptions and notation used along the paper.  Assume that the
function $0\leq f(\xi,t)$ with $(\xi,t)\in\mathbb{R}^{n}\times\mathbb{R}^{+}$ solves
the homogeneous Boltzmann problem
\begin{equation}
\frac{\partial{f}}{\partial{t}}=Q(f,f)\;\;\mathrm{on}\;(0,T)\times\mathbb{R}^{n},\;\;\;
\;f(\xi,0)=f_{0},
\label{e1}
\end{equation}
where $Q(f,f):=Q^{+}(f,f)-Q^{-}(f,f)$ is the standard Boltzmann collision operator for
variable hard potentials.  It is defined for any two measurable functions $f$ and $g$
by the formula
\begin{equation}
Q(f,g)=\int_{\mathbb{R}^{n}}\int_{S^{n-1}}\left(f'g'_{*}-fg_{*}\right)B(\xi-\xi_{*},\sigma)
d\sigma d\xi_{*}.
\label{e2}
\end{equation}
In particular,
\begin{equation}
Q^{+}(f,g)=\int_{\mathbb{R}^{n}}\int_{S^{n-1}}f'g'_{*}B(\xi-\xi_{*},\sigma)d\sigma d\xi_{*}
\label{e2.1}
\end{equation}
and,
\begin{equation}
Q^{-}(f,g)=\int_{\mathbb{R}^{n}}\int_{S^{n-1}}fg_{*}B(\xi-\xi_{*},\sigma)d\sigma
d\xi_{*} \, . \label{e2.2}
\end{equation}
\\
The classical notation $'\!f$, $'\!f_{*}$, $f'$ and $f'_{*}$ is
adopted to imply that the distributional function $f$ has the
pre-collision velocity arguments $'\!\xi$, $'\!\xi_{*}$ or the post
collision velocity arguments $\xi'$, $\xi'_{*}$.  Recall that the
dependence of post and pre-collision velocities is given by the
formulas
\begin{displaymath}
\xi'=\xi+\frac{1}{2}(|u|\sigma-u)),\;\;\;
\xi'_{*}=\xi_{*}-\frac{1}{2}(|u|\sigma-u))
\end{displaymath}
where $\sigma\in S^{n-1}$ is a vector describing the geometry of the
collisions, see for example \cite{Villani2}, for a complete
description.\\
\\
Intramolecular potentials are modeled by the collision  kernel as a
non negative function given by
\begin{displaymath}\label{collcross}
B(\xi-\xi_{*},\sigma)=|\xi-\xi_{*}|^{\alpha}\,
h(\hat{u}\cdot{\sigma}) \qquad\ \ \mathrm{and} \ \ \ \hat u =
\frac{\xi-\xi_{*}}{|\xi-\xi_{*}|}
\end{displaymath}
with $\alpha\in(0,1]$ and $\hat u$ is the renormalized relative
velocity. It is assumed that the angular cross section $h(\cdot)$
has the following properties
\begin{itemize}
\item[\it(i)] \ \ \ $h(z)\geq 0$ is nonnegative on $(-1,1)$ such that
\begin{displaymath}
h(z)+h(-z)\mbox{ is nondecreasing on }(0,1)
\end{displaymath}
\item[\it(ii)]
\begin{displaymath}
0\leq h(z)(1-z^{2})^{\mu/2}\leq C\mbox{ for }z\in(-1,1)
\end{displaymath}
\ \ \ where $\mu<n-1$ and $C>0$ constant.
\end{itemize}

\noindent Note that assumption {\textit{(i)}} implies  that
$h(\hat{u}\cdot\sigma)\in L^{1} (S^{n-1})$.  For convenience we
normalized its mass as follows
\begin{displaymath}
\int_{S^{n-1}}h(\hat u\cdot\sigma)d\sigma=\omega_{n-2}\int^{1}_{-1}h(z)(1-z^{2})^{\frac{n-3}{2}}dz=1
\end{displaymath}
where $\omega_{n-2}$ is measure of the $n-2$ dimensional sphere.
\\
In the case of  three dimensional collisional models for variable
hard potential,  condition \textit{(ii)} simplifies to
\begin{equation}
 \int^{1}_{-1}h(z)dz=1/2\pi.
\end{equation}
usually referred as the {\sl Grad cutoff assumption}.  With these
assumptions the collision model kernel used still falls in the
category of variable hard potential with angular cut-off. The reader
may go to \cite{Gamba} for a recent, complete discussion on the
behavior of the moments of the solution for variable hard potential
with cut-off in any dimension.\\
\\
The standard integrability conditions on the initial datum $f_{0}$
are assumed to be
\begin{displaymath}
\int_{\mathbb{R}^{n}}f_{0}d\xi=1,\;\;\;\;\;\int_{\mathbb{R}^{n}}f_{0}\xi
d\xi=0, \;\;\;\;\;\;\int_{\mathbb{R}^{n}}f_{0}|\xi|^2d\xi<\infty.
\end{displaymath}
In other words, $f_{0}$ has finite mass, which is normalized to one
for convenience,  and finite energy.  These conditions can be
addressed in a compact manner using the weighted Lebesgue space
$L^{p}_{k}$ with $p\geq1$ and $k\in\mathbb{R}$, defined by the norm
\begin{displaymath}
\left\|f\right\|_{L^{p}_{k}(\mathbb{R}^{n})}=\left(\int_{\mathbb{R}^{n}}
|f|^{p}(1+|\xi|^{2})^{pk/2}d\xi\right)^{1/p}.
\end{displaymath}
In particular the initial datum  can be referred as $f_{0}\in
L^{1}_{2}$.
\\
 Following these ideas, the weighted Sobolev spaces $W^{s,p}_{k}$, with $s\in N$,
 are used to work with the weak  derivatives of $f$.  These spaces are defined by the norm
\begin{displaymath}
\left\|f\right\|_{W^{s,p}_{k}(\mathbb{R}^{n})}=\left(\sum_{|\nu|\leq s}\left\|\partial^{\nu}
f\right\|^{p}_{L^{p}_{k}}\right)^{1/p}
\end{displaymath}
  where the symbol $\partial^{\eta}$ is understood as $\partial^{\eta}=
  \partial^{\eta_{1}}_{\partial \xi_{1}}\partial^{\eta_{2}}_{\xi_{2}}\cdots
  \partial^{\eta_{n}}_{\xi_{n}}$ for a multi-index $\eta$ of n dimensions.
  The usual notation is used for the Hilbert space $H^{s}_{k}\equiv W^{s,2}_{k}$.\\
  \\
Throughout the paper, the order of the multi-index is defined as
$|\eta|=~\sum_{1\leq i\leq n}\eta_{i}$,  in addition, the comparison
between multi-indexes is denoted as $\nu<\eta$ or $\nu\leq\eta$.  This is understood
as $\nu_{i}\leq\eta_{i}$ for all $1\leq i\leq n$ and $|\nu|<|\eta|$ or $|\nu|\leq|\eta|$
respectively.\\
Regarding the regularity of the initial datum, it is required that
$f_{0}\in W^{s,1}_{2}$ for some $s\geq 1$ to be chosen afterwards.
The additional assumption $f_{0}\in H^{s}$ is required for the final
result.
\begin{Definition}
Define for any sufficiently regular function $f$, multi-index $\eta$ and $p>0$
the \textsl{moment of order $p$ for the $\eta$ derivative of $f$} as the time dependent function
\begin{equation}
\delta^{\eta} m_{p}(t)\equiv\int_{\mathbb{R}^{n}}|\partial^{\eta} f||\xi|^{2p}d\xi.
\label{E3}
\end{equation}
\end{Definition}
This definition tries to generalize the classical definition for the
moments, $m_{p}(t)\equiv\int_{\mathbb{R}^{n}}f|\xi|^{2p}d\xi$,
however an absolute value is imposed in $\partial^{\eta} f$ since
this function in general does not have sign.  Observe that the
condition $f_{0}\in W^{s,1}_{2}$ is equivalent to say that
$\delta^{\nu}m_{0}(0)$ and $\delta^{\nu}m_{1}(0)$ are bounded for
$|\nu|\leq s$.\\
\\
The next definition is related to the exponential tail concept
introduced in \cite{Bobylev1} in the study of the solution's moments
of the elastic homogenous Boltzmann equation, and later in
\cite{Bobylev} in the study of large velocity tails for solutions of
the inelastic homogeneous Boltzmann equation with source terms.
\begin{Definition}
The function $f$ has \textit{an $L^{1}$ exponential (weighted) tail
of order $s>0$ in $[0,T]$} if
\begin{equation}
\bar{r}_{s}=\sup_{r>0}\left\{r:\sup_{\;0\leq t\leq
T}\left\|f\;\exp(r|\xi|^{s})\right\|_{L^{1}(\mathbb{R}^{n})}<+\infty\right\}
\label{De2}
\end{equation}
is positive.\\
In particular, for $s=2$ we simply say that $f$ has an
$L^{1}$-Maxwellian (weighted) bound.
\end{Definition}
This definition is equivalent to that one used in \cite{Bobylev1}
and \cite{Bobylev} for the solution of the homogeneous Boltzmann
equation.  Observe that the definition does not requires non
negativity in the function $f$.  This is important since the main
purpose of this paper is to study the derivatives of the solution of
problem (\ref{e1}), functions that do not have sign.\\
\\
We point out that Bobylev proved in \cite{Bobylev1} the propagation
$L^{1}$-Maxwellian tails for the hard sphere problem in three
dimentsions. That is, he showed the  existence of $L^{1}$
exponential tail of order $2$ for the solution of the homogeneous
Boltzmann equation provided the initial data has $L^{1}$ exponential
tail of order $2$ the case is special in the sense that the angular
cross section function $h(z)$ in \eqref{collcross} is constant.
Recently, this result was extended was in \cite{Gamba} under the
conditions {\textit(i-ii)} for the angular cross section, to further
show that the tail behavior is in fact `pointwise' for all times if
initially so. This fact motivates the following natural definition.
\begin{Definition}
The function $f$ has \textit{an $L^{\infty}$ exponential (weighted)
tail of order $s>0$ in $[0,T]$} if
\begin{equation}
\bar{r}_{s}=\sup_{r>0}\left\{r:\sup_{0\leq t\leq T}\left\|f\;\exp(r|\xi|^{s})\right\|_{L^{\infty}
(\mathbb{R}^{n})}\right\}<\infty
\label{De3}
\end{equation}
is positive.\\
In particular, for $s=2$ we simply say that $f$ has an
$L^{\infty}$-Maxwellian (weighted) bound.
\end{Definition}
As it was just mentioned above, the elastic, space homogeneous
Boltzmann equation for variable hard spheres (i.e. $\alpha\in(0,1]$
in equation \eqref{collcross}) it has been shown that the solution
has an $L^{\infty}$ exponential tail of order 2 in $[0,\infty)$.
This is a consequence of the rather strong fact that $L^{1}$
exponential tail implies the $L^{\infty}$ exponential tail in the
solution by means of a result like Theorem~\ref{T5} in the Appendix.
Another example of the strong relation of $L^{1}-L^{\infty}$
exponential tails is given in \cite{Mischler2}. In this work the
authors proved the existence of self-similar solutions for the
inelastic homogeneous Boltzmann equation with constant restitution
coefficient.  Using the results from \cite{Gamba}, where it was
shown proved that an steady state of a self- similar solution must
has all moments bounded and an $L^{1}$ exponential tail of order 1,
the authors went further to show the existence of such steady states
and  that also
it has an $L^{\infty}$ exponential tail of order 1.\\
\\
\bigskip
\\
We are ready to formulate the main result of this work after the
introduction of the short notation for the Maxwellian (i.e.
exponential of order 2) weight,
\begin{displaymath}
M_{r}\equiv M_{r}(\xi)=\exp(-r|\xi|^{2})\;\;\mbox{with}\;\;r\in\mathbb{R}.
\end{displaymath}
\begin{Theorem}\label{T3}
Let $\eta$ any multi-index and assume that $f_{0}\in
H^{|\eta|}_{(|\eta|-1)(1+\alpha/2)}$. In addition, assume that for
all $\nu\leq\eta$ we have that $|\partial^{\nu}f_{0}|/M_{r_{0}}\in
L^{1}$ and
$|\partial^{\nu}f_{0}|/\left\{(1+|\xi|^{2})^{|\nu|/2}M_{r_{0}}\right\}\in
L^{\infty}$ for some $r_{0}>0$.  Then, there exist $r\leq r_{0}$
such that
\begin{displaymath} \sup_{\;t\geq
0}\frac{|\partial^{\nu}f|}{(1+|\xi|^{2})^{|\nu|/2}M_{r}}\leq
K_{\eta,r_{0}}
\end{displaymath}
for all $\nu\leq\eta$, where $K_{\eta,r_{0}}$ is a positive constant
depending on $\eta$, $r_{0}$ and the kernel $h(\cdot)$.  In
particular, for $\nu\leq\eta$ and $t>0$
\begin{displaymath}
\lim_{|\xi|\rightarrow\infty}|\partial^{\nu} f(\xi,t)|\leq
K_{\eta,r_{0}}
\lim_{|\xi|\rightarrow\infty}M_{\bar{r}^{\nu}_{2}}(\xi)
\end{displaymath}
where the constant $\bar{r}^{\nu}_{2}$ is given by (\ref{De3}) for
the function $\partial^{\nu}f$.
\end{Theorem}
$Remarks:$
\begin{itemize}
\item In other words, if $\partial^{\nu}f_{0}$ has a $L^{\infty}$ exponential tail of order 2
for all $\nu\leq\eta$, then the $\nu$ derivative of the solution
will propagate such behavior, that is the $\partial^{\nu}f(t,v)$
still has an $L^{1}$ exponential tail of order 2. In addition, using
related arguments to the ones in \cite{Gamba}, yields the
propagation of $L^{\infty}$ exponential tails of order 2 for  the
$\nu$ derivative of the solution, for all $\nu\leq\eta$.

\item It is clear that the property of having $L^{1}$ or $L^{\infty}$ exponential tail is
transparent to the polynomial weight that we include.  Indeed, a
function has any of the previous properties if an only if the
product of the function with a polynomial also has the property.  We
include the weight in the statement of Theorem~\ref{T3} since it
appears naturally for  variable hard sphere kernels with an angular
cross section function $h(z)$ satisfying {\textit{(i)-(ii)}},  as
the proof of the Theorem shows. In addition, emphasis has been done
about the fact that the $\nu$ derivative of the solution is being
compared with the $\nu$ derivative of the Maxwellian.
\end{itemize}
\medskip
 As it was noticed in  \cite{Bobylev1} and
\cite{Bobylev}, the existence of $L^{1}$ exponential tails for a
solution  $f$ of the space homogeneous Boltzmann equation, is
closely related to the existence of all its moments and its
summability properties. Following that line of work in the such of
properties of such nature for $\partial^{\eta} f$, we also observe
that
\begin {displaymath}\label{summ}
\int_{\mathbb{R}^{n}}|\partial^{\eta}
f|\;\exp{(r|\xi|^{s})d\xi=\sum_{k=0}^{\infty}\frac{\delta
^{\eta}m_{sk/2}}{k!}r^{k}}.
\end{displaymath}
Thus, in order to show that there is a the choice of $s>0$ for which
the summability of moments, or equivalently, a bound for the right
hand side of \eqref{summ} uniformly in time $t$, one would need to
show that there exist positive constants $K$ and $Q$, independent of
$t$, such that $\delta^{\eta} m_{sk/2}/k!<KQ^{k},\;\;k=1,2,3\cdots$.
Hence, if that is the case, the sum in the right hand side of
\eqref{summ} converges choosing, uniformly in time, for  any $r$
such that $0<r<1/Q$. This fact  will imply that the integral is
finite and therefore $\bar{r}^{\eta}_{s}>0$. \\
  Bobylev et al,
in \cite{Bobylev1}, and in \cite{Bobylev}, proved that under precise
conditions the moments of $f$ satisfy estimates
\begin{equation}
m_{sk/2}/k!<KQ^{k} \ \ \ \mathrm{uniformly\ in\ time\ for }\ s=2  \
.
\end{equation}
This paper intends to do the same for $\delta^{\eta} m_{sk/2}$ as defined in \eqref{E3}.\\
\\
Conversely, if the integral in the left hand side is bounded on
$[0,T]$ for some positive $r,s$, the terms in the sum must be
controlled in the form $\delta^{\eta}
m_{sk/2}/k!<K\;Q^{k},\;\;\;k=1,2,3\cdots$ for some constants
$K,Q>0$.  Thus, the moments $\delta^{\eta} m_{sk/2}$ with
$k=1,2,3,\cdots$ are uniformly bounded on $[0,T]$ if and only if
$\partial^{\eta}f$ has an $L^{1}$ exponential tail of some order
$s>0$ in $[0,T]$.
\\
Before continuing with the technical work the
reader may go to Appendix A. and see some of the classical results
known for a distributional solution $f$ of (\ref{e1}) used
throughout this work.
%
\section{Sharp Povzner-type inequalities for the solution's derivatives}
The purpose of this section is to give technical Lemmas regarding the derivative
of the collision operator $\partial^{\eta} Q(f,f)$.  The idea of the following Lemmas is to
obtain expressions for this operator as close as possible to those already given in
\cite{Bobylev} for $Q(f,f)$.
\begin{Lemma}\label{L1}
Let $f$ a sufficiently smooth function.  Then, the following expressions hold for the
positive and negative parts of the collision operator
\begin{displaymath}
\partial^{\eta} Q^{\pm}(f,f)=\sum_{\nu\leq\eta}\left(\begin{array}{c}\eta\\\nu\end{array}\right)
Q^{\pm}(\partial^{\nu}f,\partial^{\eta-\nu}f).
\end{displaymath}
In particular,
\begin{displaymath}
\partial^{\eta} Q(f,f)=\sum_{\nu\leq\eta}\left(\begin{array}{c}\eta\\\nu\end{array}\right)
Q(\partial^{\nu}f,\partial^{\eta-\nu}f).
\end{displaymath}
\end{Lemma}
\begin{proof}
This is a direct consequence of the invariance property
$\tau_{\Delta}Q(f,f)= Q(\tau_{\Delta}f,\tau_{\Delta}f)$, where
$\tau_{\Delta}$ is the translation operator defined by
$\tau_{\Delta}g(\xi)=g(\xi-\Delta)$, for $\xi$ and $\Delta$ in
$\mathbb{R}^{n}$.   For details see \cite{Villani}.
\end{proof}
Next,
we need a suitable form for the action of the derivative of the
collision operator $\partial^{\nu}Q(f,f)$ on test functions.
\begin{Lemma}\label{L2}
Let $f$ a sufficiently smooth function.  Then, the action of the
$\eta$ derivative of the collision operator on any test function
$\phi$ is given by
\begin{multline}\label{e3}
\int_{\mathbb{R}^{n}}\partial^{\eta}Q(f,f)\phi
d\xi=\int\int_{\mathbb{R}^{n}\times\mathbb{R}^{n}}
f_{*}\,\partial^{\eta}f\, A[\phi]|u|^{\alpha}d\xi_{*}d\xi\\+
1/2\sum_{0<\nu<\eta}\left(\begin{array}{c}\eta\\\nu
\end{array}
\right)\int\int_{\mathbb{R}^{n}\times\mathbb{R}^{n}}
\partial^{\nu}f \, \partial^{\eta-\nu}f_{*}\, A[\phi]|u|^{\alpha}d\xi_{*}d\xi
\end{multline}
where
\begin{displaymath}
A[\phi]\ =\ A^{+}[\phi]\ -\ A^{-}[\phi]
\end{displaymath}
with
\begin{equation}
A^{+}[\phi]\ =\
\int_{S^{n-1}}(\phi'+\phi'_{*})h(\hat{u}\cdot\sigma)d\sigma\ \ \ \
\mathrm{\  and \  }\ \ \ \ A^{-}[\phi]\ =\ \phi+\phi_{*}\label{e4}
\end{equation}
\end{Lemma}
\begin{proof}
For $f$ and $g$ any smooth functions we have after the regular
change of variables $\xi\rightarrow\xi^{'}$
\begin{equation}
\int_{\mathbb{R}^{n}} Q^{+}(f,g)\phi
d\xi=\frac{1}{2}\int\int\int_{\mathbb{R}^{n}\times\mathbb{R}^{n}\times{S^{n-1}}}
(fg_{*}\phi'+f_{*}g\phi'_{*})h(\hat{u}\cdot\sigma)d\sigma
|u|^{\alpha}d\xi_{*}d\xi. \label{e5}
\end{equation}
Also, using the change of variables $\xi\rightarrow\xi_{\ast}$ the
action of the negative collision part is given by
\begin{equation}
\int_{\mathbb{R}^{n}} Q^{-}(f,g)\phi
d\xi=\frac{1}{2}\int\int_{\mathbb{R}^{n}\times\mathbb{R}^{n}}
(fg_{*}\phi+f_{*}g\phi_{*})|u|^{\alpha}d\xi_{*}d\xi. \label{e6}
\end{equation}
Now, using Lemma~\ref{L1}
\begin{displaymath}
\int_{\mathbb{R}^{n}} \partial^{\eta}Q^{+}(f,f)\phi
d\xi=\sum_{\nu\leq\eta}\left(\begin{array}{c}\eta\\\nu
\end{array}\right)\int_{\mathbb{R}^{n}}Q^{+}(\partial^{\nu}f,\partial^{\eta-\nu}f)\phi
d\xi.
\end{displaymath}
Let $f\equiv\partial^{\nu}f$ and
$g\equiv\partial^{\eta-\nu}f$ in (\ref{e5}) to get
\begin{equation}
\int_{\mathbb{R}^{n}} \partial^{\nu}Q^{+}(f,f)\phi
d\xi=1/2\sum_{\nu\leq\eta}\left(\begin{array}{c}\eta\\\nu\end{array}\right)
\int\int_{\mathbb{R}^{n}\times\mathbb{R}^{n}}\partial^{\nu}f\,
\partial^{\eta-\nu} f_{*}\, A^{+}[\phi]|u|^{\alpha}d\xi_{*}d\xi.
\label{e7}
\end{equation}
Following the same idea, and using (\ref{e6}) and the renormalization of
angular cross section
\begin{equation}
\int_{\mathbb{R}^{n}} \partial^{\eta}Q^{-}(f,f)\phi
d\xi=1/2\sum_{\nu\leq\eta}
\left(\begin{array}{c}\eta\\\nu\end{array}\right)\int\int_{\mathbb{R}^{n}\times
\mathbb{R}^{n}}\partial^{\nu}f\,
\partial^{\eta-\nu}f_{*}\, A^{-}[\phi]|u|^{\alpha}d\xi_{*}d\xi.
\label{e8}
\end{equation}
Subtract (\ref{e8}) from (\ref{e7}) and split the total sum to conclude.
\end{proof}
The moments of the derivative of the collision operator need to be
controlled in order to find a bound for the moments of the
solution's derivatives.  The following Lemma is a first step in this
direction.
\begin{Lemma}\label{L3}
Assume $\phi\geq0$, then for any multi-index $\eta$
\begin{multline}\label{e9}
\int_{\mathbb{R}^{n}}\partial^{\eta}Q(f,f)\mathrm{sgn}(\partial^{\eta}f)\phi d\xi\leq\\
\int\int_{\mathbb{R}^{n}\times\mathbb{R}^{n}}f_{*}|\partial^{\eta} f|A[\phi]|u|^{\alpha}
d\xi_{*}d\xi+ 2\int\int_{\mathbb{R}^{n}\times\mathbb{R}^{n}}f_{*}|\partial^{\eta}
f|\phi_{*}|u|^{\alpha}d\xi_{*}d\xi\\+1/2\sum_{0<\nu<\eta}\left(\begin{array}{c}\eta\\
\nu\end{array}\right)\int\int_{\mathbb{R}^{n}\times\mathbb{R}^{n}}|\partial^{\nu}
f\partial^{\eta-\nu}f_{*}|A[\phi]|u|^{\alpha}d\xi_{*}d\xi\\+\sum_{0<\nu<\eta}
\left(\begin{array}{c}\eta\\\nu\end{array}\right)\int\int_{\mathbb{R}^{n}\times
\mathbb{R}^{n}}|\partial^{\nu}f\partial^{\eta-\nu}f_{*}|\phi_{*}|u|^{\alpha}d\xi_{*}d\xi
\end{multline}
\end{Lemma}
\begin{proof}
Let $\Psi\geq 0$ and  $\phi=\mathrm{sgn}(\partial^{\eta}f)\Psi$ in Lemma~\ref{L2}.
In one hand,  observe that for the first term in (\ref{e3})  $\left|A^{+}[\mathrm{sgn}
(\partial^{\eta}f)\Psi]\right|\leq A^{+}[\Psi]$, hence
\begin{displaymath}
\partial^{\eta}fA^{+}[\mathrm{sgn}(\partial^{\eta} f)\Psi]\leq|\partial^{\eta}f|A^{+}[\Psi].
\end{displaymath}
On the other hand
\begin{displaymath}
\partial^{\eta}fA^{-}[\mathrm{sgn}(\partial^{\eta} f)\Psi]=|\partial^{\eta}f|A^{-}
[\Psi]-\Psi_{*}\left\{|\partial^{\eta}f|-\mathrm{sgn}(\partial^{\eta}f_{*})
\partial^{\eta}f\right\}.
\end{displaymath}
Gathering these two inequations we have
\begin{equation}
\partial^{\eta}fA[\mathrm{sgn}(\partial^{\eta}f)\Psi]\leq|\partial^{\eta}
f|A[\Psi]+2|\partial^{\eta}f|\Psi_{*}.
\label{e10}
\end{equation}
Note that (\ref{e10}) yields a control for the first term in
(\ref{e3}) of Lemma~\ref{L2}.  Moreover, a similar argument also
works for the second term in (\ref{e3}).
\end{proof}
Although the expression (\ref{e9}) may look cumbersome, we point out
that the main idea here is to separate the terms that depend on the
actual derivative of order $\eta$ from the lower order derivatives.
In this way, it is possible to take advantage of the expression
(\ref{e9}) when an induction argument is used.\\
\\
We are now ready to study the moments of the solution's derivatives
for $p>1$.  The idea is to follow the work \cite{Bobylev} adapting
the results to this extended case.  Several Lemmas are needed before
attempting to prove a time uniform control on these moments. Let us
first consider, in the following Lemma, test functions of the form
$\phi_{p}=|\xi|^{2p}$, with $p>1$.  For a detailed proof see
\cite{Bobylev} for bounded angular cross section function $h(z)$ and
from \cite{Gamba} for $h(z)$ satisfying conditions
\textit{(i)-(ii)}.  Nevertheless,  we present a slightly modified
argument from the one in \cite{Gamba} to handle condition
\textit{(ii)}:
\begin{Lemma}\label{L4}
Under the previous assumptions on $h(\cdot)$, for every $p\geq 1$,
\begin{equation}
A[\phi_{p}]=A[|\xi|^{2p}]\leq
-(1-\gamma_{p})(|\xi|^{2p}+|\xi_{*}|^{2p})+\gamma_{p}((|\xi|^{2}+
|\xi_{*}|^{2})^{p}-|\xi|^{2p}-|\xi_{*}|^{2p})
\label{e11}
\end{equation}
where the constant $\gamma_{p}$ is given by the formula
\begin{equation}\label{ga}
\gamma_{p}=\omega_{n-2}\int^{1}_{-1}\left(\frac{1+z}{2}\right)^{p}\bar{h}(z)
(1-z^{2})^{\frac{n-3}{2}}dz
\end{equation}
with $\bar{h}(z)=\frac{1}{2}(h(z)+h(-z))$.
 In particular, for  $\epsilon=n-1-\mu>0$
\begin{equation}\label{ga1}
 \lim_{p\rightarrow\infty}\gamma_{p} \sim p^{-\epsilon/2}
 \searrow0 \, ,
 \end{equation}
  where $\mu$ is the growth exponent of condition
 \textit{(ii)} on $h(z)$.
 \\
Furthermore if $h(z)$ is bounded, the following estimate holds for
for $p>1$
\begin{displaymath}
\gamma_{p}<\min\left\{1,\frac{16\pi\left\|h\right\|_{\infty}}{p+1}\right\}
\, .
\end{displaymath}
\end{Lemma}
\begin{proof}
It is easy to see that $
\lim_{p\rightarrow\infty}\gamma_{p}\searrow0$, since  by conditions
\textit{(i-ii)} in $h(z)$ it follows that $\gamma_{1}$ is bounded.
In particular
\begin{displaymath}
\left(\frac{1+z}{2}\right)^{p}\searrow 0\;\;\mbox{a.e.
in}\;\;(-1,1)\;\;\mbox{as}\;\;p\rightarrow\infty
\end{displaymath}
so $\gamma_{p}$ is decreasing on $p$.  Using Lebesgue's Dominated
Convergence the decay of $\gamma_{p}\searrow0$ follows.
\\
However, using \textit{(ii)} on $h(z)$ we can say more about the
decreasing rate of $\gamma_{p}$ to zero. Since
\begin{displaymath}
\bar{h}(z)\leq C(1-z^{2})^{-\mu/2}
\end{displaymath}
then
\begin{align}\label{e11.1}
\gamma_{p}&=\omega_{n-2}\int^{1}_{-1}\left(\frac{1+z}{2}\right)^{p}
\bar{h}(z)(1-z^{2})^{\frac{n-3}{2}}dz\nonumber\\ &\leq 2^{\epsilon-1}C\omega_{n-2}
\int^{1}_{0}s^{p+\epsilon/2-1}(1-s)^{\epsilon/2-1}ds\nonumber\\
&=2^{\epsilon-1}C\omega_{n-2}\beta(p+\epsilon/2,\epsilon/2)
\end{align}
where $\epsilon=n-1-\mu>0$.  Then we can estimate
\begin{displaymath}
\beta(p+\epsilon/2,\epsilon/2)=\frac{\Gamma(p+\epsilon/2)\Gamma(\epsilon/2)}{\Gamma(p+\epsilon)}\sim
p^{-\epsilon/2}\;\;\mbox{for large}\;\;p.
\end{displaymath}
It is concluded that $\gamma_{p}\sim p^{-\epsilon/2}$ when
$p\rightarrow\infty$ and \eqref{ga1} holds.
\end{proof}
$Remarks:$
\begin{itemize}
\item Lemma \ref{L4} is a sharp version of the so called Povzner inequalities.
It is proved after a clever manipulation of the post collision variables in the
positive part of the collision operator.  The convexity of $h(\cdot)$ plays an
essential role to conclude Lemma (\ref{L4}) because it provides a non decreasing
property to the action of the gain operator on convex functions. \item For hard
spheres $h(\hat{u}\cdot\sigma)=1/4\pi$, hence $\gamma_{p}<\min\left\{1,\frac{4}{p+1}\right\}$.
\end{itemize}
\begin{Lemma}\label{L5}
Assume that $p>1$, and let $k_{p}=[\frac{p+1}{2}]$.  Then for all
$x,y>0$ the following inequalities hold
\begin{equation}
\sum^{k_{p}-1}_{k=1}\left(\begin{array}{c}p\\k\end{array}\right)(x^{k}y^{p-k}+x^{p-k}y^{k})
\leq(x+y)^{p}-x^{p}-y^{p}\leq\sum^{k_{p}}_{k=1}
\left(\begin{array}{c}p\\k\end{array}\right)(x^{k}y^{p-k}+x^{p-k}y^{k}).
\label{e12}
\end{equation}
\end{Lemma}
$Remark:$ The binomial coefficient for a non integer $p$ is defined for $k\geq1$ as
\begin{equation}
\left(\begin{array}{c}p\\k\end{array}\right)=\frac{p(p-1)\cdots(p-k+1)}{k!}\mbox{, and }
\left(\begin{array}{c}p\\0\end{array}\right)=1.
\label{e13}
\end{equation}
The following Lemma is a consequence of the previous two results.
It provides a control on the collision operator's moments of order $p$ using moments
of the solution's derivatives of order strictly less that $p$.
\begin{Lemma}\label{L6}
Assume $h(z)$ fulfill all the conditions discussed, then for every $p>1$ and
multi-index $\eta$
\begin{multline}
\int_{\mathbb{R}^{n}} \partial^{\eta}Q(f,f)sgn(\partial^{\eta} f)|\xi|^{2p}d\xi
\leq-(1-\gamma_{p})k_{\alpha}\delta^{\eta}m_{p+\alpha/2}+\gamma_{p}\delta^{\eta}
S_{p}\\+\delta^{\eta^{-}}(m_{\alpha/2}m_{p})+\delta^{\eta^{-}}(m_{0}m_{p+\alpha/2})
\label{e14}
\end{multline}
where $k_{\alpha}$ is a positive constant depending on $\alpha$ but not on $p$.\\
In addition,
\begin{displaymath}
\delta^{\eta}S_{p}\equiv\sum^{k_{p}}_{k=1}\left(\begin{array}{c}p\\k\end{array}\right)
\left\{\delta^{\eta}(m_{k}m_{p-k+\alpha/2})+\delta^{\eta}(m_{k+\alpha/2}m_{p-k})\right\}
\mbox{ with }k_{p}=\left[\frac{p+1}{2}\right]
\end{displaymath}
and
\begin{displaymath}
\delta^{\eta^{-}}(m_{\alpha/2}m_{p})+\delta^{\eta^{-}}(m_{0}m_{p+\alpha/2})\equiv2
\sum_{\nu<\eta}\left(\begin{array}{c}\eta\\\nu\end{array}\right)\left\{\delta^{\eta-\nu}
m_{\alpha/2}\delta^{\nu}m_{p}+\delta^{\eta-\nu}m_{0}\delta^{\nu}m_{p+\alpha/2}\right\}.
\end{displaymath}
\end{Lemma}
$Remarks:$
\begin{itemize}
\item The notation
\begin{displaymath} \delta^{\eta}(m_{p}m_{q})\equiv\sum_{\nu\leq\eta}
\left(\begin{array}{c}\eta\\\nu\end{array}\right)\delta^{\nu} m_{p}
\delta^{\eta-\nu}m_{q}\:\:\:\mbox{and}
\end{displaymath}
\begin{equation}\label{e14.1}
\delta^{\eta^{-}}(m_{p}m_{q})\equiv\sum_{\nu<\eta}\left(\begin{array}{c}
\eta\\\nu\end{array}\right)\delta^{\nu}
m_{p}\delta^{\eta-\nu}m_{q}
\end{equation}
has been chosen so that the ``product rule of differentiation"
holds. Expression (\ref{e14}) makes it clear that this notation is
very convenient to maintain the length of expressions short and at
the same time easy to remember. The minus sign next to the upper
script $\eta$ in the last term of (\ref{e14}) was introduced to
stress the fact that the sum is done on the multi-index $\nu<\eta$.
\item Observe that none of the two last terms in (\ref{e14}) depends
on $\delta^{\eta}m_{p+\alpha/2}$ for $p>1$.  This is important for the
induction arguments used later on.
\end{itemize}
\begin{proof}
Let $\phi=|\xi|^{2p}$ in Lemma~\ref{L3}.  The sum of terms two and
four in the right-hand side of (\ref{e9}) is bounded by
$\delta^{\eta^{-}}(m_{\alpha/2}
m_{p})+\delta^{\eta^{-}}(m_{0}m_{p+\alpha/2})$.  Indeed, use the
inequality $|u|^{\alpha}\leq|\xi|^{\alpha}+|\xi_{*}|^{\alpha}$,
which is valid for $\alpha\in(0,1]$, expand the integrals, and use
the definition of moments of the derivatives (\ref{E3}) to conclude
directly.
\\
Recall the inequality, found in \cite{Ark} or \cite[Lemma
7]{Gamba}, valid for solutions of the Boltzmann equation with finite
mass, energy and entropy,
\begin{equation}\label{e14.8}
\int_{\mathbb{R}^{n}}f_{*}|u|^{\alpha}d\xi_{*}\geq k_{\alpha}|\xi|^{\alpha}
\end{equation}
where the constant $k_{\alpha}$ depends, in addition to $\alpha$, on the mass,
energy and entropy of the initial datum $f_{0}$.
Thus, it follows that
\begin{align}
\int\int_{\mathbb{R}^{n}\times\mathbb{R}^{n}}f_{*}|\partial^{\eta}f|(|\xi|^{2p}+
|\xi_{*}^{2p}|)|u|^{\alpha}d\xi_{*} d\xi&\geq\int\int_{\mathbb{R}^{n}\times\mathbb{R}^{n}}
f_{*}|\partial^{\eta}f||\xi|^{2p}|u|^{\alpha}d\xi_{*}d\xi\nonumber\\
&\geq k_{\alpha}\delta^{\eta}m_{p+\alpha/2}.
\label{e14.9}
\end{align}
Use (\ref{e14.9}), Lemma~\ref{L4} and Lemma~\ref{L5}, to control the first term
in the right hand side of (\ref{e9})
\begin{multline*}
\int\int_{\mathbb{R}^{n}\times\mathbb{R}^{n}} f_{*}|\partial^{\eta}
f|A[\phi]|u|^{\alpha}
d\xi_{*}d\xi\leq-(1-\gamma_{p})k_{\alpha}\delta^{\eta}m_{p+\alpha/2}+\\\gamma_{p}
\sum^{k_{p}}_{k=1}\left(\begin{array}{c}p\\k\end{array}\right)\delta^{\eta}m_{k+\alpha/2}
m_{p-k}+\delta^{\eta}m_{p-k}m_{k+\alpha/2}+\delta^{\eta}m_{k}m_{p-k+\alpha/2}+\delta^{\eta}
m_{p-k+\alpha/2}m_{k}.\end{multline*}
Similarly, the following control is obtained for the third term of
(\ref{e9})
\begin{multline*}
\sum_{0<\nu<\eta}\left(\begin{array}{c}\eta\\\nu\end{array}\right)\int\int_{\mathbb{R}^{n}
\times\mathbb{R}^{n}} |\partial^{\nu}f\partial^{\eta-\nu}f_{*}|A[\phi]|u|^{\alpha}d\xi_{*}
d\xi\leq\\\gamma_{p}\sum^{k_{p}}_{k=1}\left(\begin{array}{c}p\\k\end{array}\right)
\sum_{0<\nu<\eta}\left(\begin{array}{c}\eta\\\nu\end{array}\right)\delta^{\nu}m_{k+\alpha/2}
\delta^{\eta-\nu}m_{p-k}+\delta^{\nu}m_{p-k}\delta^{\eta-\nu}m_{k+\alpha/2}\\+\delta^{\nu}m_{k}
\delta^{\eta-\nu}m_{p-k+\alpha/2}+\delta^{\nu}m_{p-k+\alpha/2}\delta^{\eta-\nu}m_{k}\;.
\end{multline*}
Combining these two inequalities, the sum of first and third term of (\ref{e9}) is bounded by
\begin{displaymath}
 -(1-\gamma_{p})k_{\alpha}\delta^{\eta}m_{p+\alpha/2}+\gamma_{p}\delta^{\eta}S_{p}.
\end{displaymath}
This concludes the proof.
\end{proof}
Let us introduce the normalized moments, which are defined as
\begin{equation}
\delta^{\eta}z_{p}\equiv\delta^{\eta} m_{p}/\Gamma(p+b)
\label{E15}
\end{equation}
where $\Gamma(\cdot)$ is the gamma function and $b>0$ is a free
parameter to be chosen in the sequel.  These moments can be used to
simplify the estimate obtained in Lemma (\ref{L6}).
\\
\begin{Lemma}\label{L7}
For every $p>1$ and multi-index $\eta$,
\begin{displaymath}
\delta^{\eta}S_{p}\leq A\;\Gamma(p+\alpha/2+2b)\delta^{\eta}Z_{p}
\end{displaymath}
where
\begin{equation}
\delta^{\eta}Z_{p}=\max_{1\leq k\leq k_{p}}\left\{\delta^{\eta}(z_{k}z_{p-k+\alpha/2}),
\delta^{\eta}(z_{k+\alpha/2}z_{p-k})\right\}
\label{e15}
\end{equation}
and $A>0$ is a constant depending only on $b$.
\end{Lemma}
\begin{proof}
The proof is identical to that of Lemma 4 in \cite{Bobylev}.  First observe that
\begin{multline*}
\delta^{\eta}S_{p}=\sum^{k_{p}}_{k=1}\left(\begin{array}{c}p\\k\end{array}\right)
\Gamma(k+b)\Gamma(p-k+\alpha/2+b)\delta^{\eta}(z_{k}z_{p-k+\alpha/2})\\+\Gamma(k+\alpha/2+b)
\Gamma(p-k+b)\delta^{\eta}(z_{k+\alpha/2}z_{p-k}).
\end{multline*}
But the Beta and Gamma functions are related by
\begin{displaymath}
\beta(x,y)=\frac{\Gamma(x)\Gamma(y)}{\Gamma(x+y)}.
\end{displaymath}
This allow us to reduce the right-hand side in the previous equality to
\begin{multline*}
\Gamma(p+\alpha/2+2b)\sum^{k_{p}}_{k=1}\left(\begin{array}{c}p\\k\end{array}\right)
\beta(k+b,p-k+\alpha/2+b)\delta^{\eta}(z_{k}z_{p-k+\alpha/2})\\+\beta(k+\alpha/2+b,p-k+b)
\delta^{\eta}(z_{k+\alpha/2}z_{p-k}).
\end{multline*}
Therefore,
\begin{multline}
\delta^{\eta}S_{p}\leq\Gamma(p+\alpha/2+2b)\;\delta^{\eta}Z_{p}\;\sum^{k_{p}}_{k=1}
\left(\begin{array}{c}p\\k\end{array}\right)\beta(k+b,p-k+\alpha/2+b)\\+\beta(k+\alpha/2+b,p-k+b).
\label{e15.9}
\end{multline}
Using the definition of the beta function it is possible to control the sum in (\ref{e15.9})
by constant $A$ depending only on $b$, for details of this last step see~\cite[Lemma 4]{Bobylev}.
\end{proof}
We are ready to construct the system of differential inequalities used in the proof of
Theorem~\ref{T1}.  The following Lemma follows by a direct application of Lemma~\ref{L6}
and Lemma~\ref{L7} on the Boltzmann equation (\ref{e1}).
\begin{Lemma}\label{L8}
Let $\eta$ any multi-index and assume that $\delta^{\eta}m_{0}>0$ and $\delta^{\nu}m_{0},
\;\delta^{\nu}m_{\alpha/2}$ uniformly bounded on time for all $\nu\leq\eta$, then
\begin{multline}\label{e16}
\frac{d(\delta^{\eta}z_{p})}{dt}+(1-\gamma_{p})k_{\alpha}\Gamma(p+b)^{\alpha/2p}
\delta^{\eta}m_{0}^{-\alpha/2p}(\delta^{\eta}z_{p})^{1+\alpha/2p}\leq\gamma_{p}
k_{0}p^{\alpha/2+b}\;\delta^{\eta}Z_{p}\\+k_{1}p^{\alpha/2}\delta^{\eta^{-}}(m_{0}z_{p+\alpha/2})
+\delta^{\eta^{-}}(m_{\alpha/2}z_{p})
\end{multline}
for all $p>1$, with $k_{1}>0$ universal constant, $k_{0}>0$ depending only on $b$ and
$k_{\alpha}$ given in Lemma (\ref{L6})
\end{Lemma}
\begin{proof}
First, note that using Jensen's inequality
\begin{displaymath}
\delta^{\eta}m_{p+\alpha/2}\geq\delta^{\eta}m_{0}^{-\alpha/2p}\delta^{\eta}m_{p}^{1+\alpha/2p}.
\end{displaymath}
Next, take the $\eta$ derivative in velocity in both sides of the
Boltzmann equation (\ref{e1}), then multiply it by
$\mathrm{sgn}(\partial^{\eta}f)|\xi|^{2p}$ and integrate in
velocity, then, use Lemma~\ref{L6} to obtain
\begin{multline*}
\frac{d(\delta^{\eta}m_{p})}{dt}+(1-\gamma_{p})k_{\alpha}\delta^{\eta}m_{0}^{-\alpha/2p}
(\delta^{\eta}m_{p})^{1+\alpha/2p}\leq\\
\gamma_{p}\delta^{\eta}S_{p}+\left\{\delta^{\eta^{-}}(m_{p+\alpha/2}m_{0})+\delta^{\eta^{-}}
(m_{p}m_{\alpha/2})\right\}.
\end{multline*}
Use the definition of $\delta^{\eta}z_{p}$ in the previous
inequality, and combine it with Lemma~\ref{L7} to
get
\begin{multline*}
\frac{d(\delta^{\eta}z_{p})}{dt}+(1-\gamma_{p})k_{\alpha}\Gamma(p+b)^{\alpha/2p}\delta^{\eta}
m_{0}^{-\alpha/2p}(\delta^{\eta}z_{p})^{1+\alpha/2p}\leq\\\gamma_{p}\;K\frac{
\Gamma(p+\alpha/2+2b)}{\Gamma(p+b)}\delta^{\eta}Z_{p}+\frac{\Gamma(p+\alpha/2+b)}
{\Gamma(p+b)}\delta^{\eta^{-}}(m_{0}z_{p+\alpha/2})+\delta^{\eta^{-}}(m_{\alpha/2}z_{p}).
\end{multline*}
For the terms involving Gamma functions use the asymptotic formulas
for large $p$
\begin{displaymath}
\frac{\Gamma(p+\alpha/2+2b)}{\Gamma(p+b)}\sim p^{\alpha/2+b},\;\;\;\frac{\Gamma(p+\alpha/2+b )}
{\Gamma(p+b)}\sim p^{\alpha/2}
\end{displaymath}
to find the right polynomial grow.
\end{proof}
$Remark:$ Lemma~\ref{L8} is the equivalent result to Lemma 6.2 in \cite{Bobylev1}.
 Differential inequalities (\ref{e16}) have the additional complication that they are not
 of constant coefficients since $\delta^{\eta}m_{0}$, unlike $m_{0}$, is in general a
 function of $t$.
 \\ The following classical result on ODE's helps to implement an induction argument on
 inequalities of the form (\ref{E20}-\ref{E21}).
\begin{Lemma}\label{L9}
Let $\textbf{a}$ and $\textbf{b}$ positive continuous functions in $t$ such that
\begin{displaymath}
\textbf{a}_{*}=\inf_{\;t>0}\;\textbf{a}>0,\;\;\;\textbf{b}_{*}=\sup_{\;t>0}\;\textbf{b}<+\infty
\end{displaymath}
and let $\textbf{c}$ a positive constant.  In addition assume that $y\geq0\in C^{1}([0,\infty))$
solves the differential inequality
\begin{equation}
y'+\textbf{a}y^{1+\textbf{c}}\leq \textbf{b},\;\;\;\;y(0)=y_{0}
\label{e17}
\end{equation}
then $y\leq \max\left\{y_{0},(\textbf{b}_{*}/\textbf{a}_{*})^{1/(1+c)}\right\}$
\end{Lemma}
\begin{proof}
Since $y'+\textbf{a}_{*}y^{1+\textbf{c}}\leq \textbf{b}_{*}$, it
suffices to prove the Lemma in the case that $\textbf{a}$ and
$\textbf{b}$ are positive constants.   As a first step assume
equality in (\ref{e17}).  Thus, from classical theory of
differential equations this ODE has a unique $C^{1}$ solution
$y_{*}$ with the property stated by the Lemma, i.e.
\begin{displaymath} y_{*}\leq
\max\left\{y_{0},(\textbf{b}/\textbf{a})^{1/(1+c)}\right\}.
\end{displaymath}
If $y\in C([0,\infty))$ solves (\ref{e17}) we claim that $y\leq
y_{*}$.  Assume that there exist $T'>0$ where the inequality
$y(T')>y_{*}(T')$ holds. Let $T<T'$ such that $y(T)=y_{*}(T)$ and
$y>y_{*}$ in $(T,T')$.  The existence of such a point is assured by
the continuity of the functions in addition to the fact that
$y(0)=y_{*}(0)=y_{0}$. Therefore,
\begin{displaymath}
\int^{T'}_{T}y'=y(T')-y(T)>y_{*}(T')-y_{*}(T)=\int^{T'}_{T}y'_{*}.
\end{displaymath}
Hence, there exist $T_{0}\in(T,T')$ such that
$y'(T_{0})>y'_{*}(T_{0})$.  Thus, the inequalities
\begin{displaymath} \textbf{b}-y(T_{0})^{1+\textbf{c}}\geq
y'(T_{0})>y'_{*}(T_{0})=\textbf{b}-y_{*}(T_{0})^{1+\textbf{c}}
\end{displaymath}
hold.  Consequently, it is concluded that $y(T_{0})<y_{*}(T_{0})$
which is a contradiction. As a result, $y\leq y_{*}$. This proves
the Lemma. \end{proof} $Remark:$ Same argument proves a similar
result for $y\geq0\in C^{1}([0,\infty))$ solving the differential
inequality $y'+\textbf{a}y^{1+\textbf{c}}\leq
\textbf{d}y+\textbf{b}$ with $\textbf{a}$, $\textbf{b}$ and
$\textbf{d}$ positive function on $t$ (satisfying similar
hypothesis) and $\textbf{c}$ a positive constant.  Of course the
bound slightly changes to $y\leq \max\left\{y_{0},\bar{y}\right\}$
where point $\bar{y}$ is given by the equation
$\textbf{a}_{*}\bar{y}^{1+\textbf{c}}=\textbf{d}_{*}\bar{y}+\textbf{b}_{*}$.
Here $\textbf{d}_{*}=\sup_{\;t>0}\;\textbf{d}$.\\
\\
Next result relate the last two Lemmas and gives some orientation
of the future application of Lemma~\ref{L9}.
\begin{Corollary}\label{C1}
Inequalities (\ref{e16}) can be written for $p=3/2$ as
\begin{equation}
\frac{d(\delta^{\eta}z_{3/2})}{dt}+\textbf{a}_{3/2}(\delta^{\eta}z_{3/2})^{1+\textbf{c}_{3/2}}
\leq \textbf{b}_{3/2}+\textbf{d}_{3/2}(\delta^{\eta}z_{3/2})
\label{E20}
\end{equation}
and for $p\in\left\{2,5/2,3,...\right\}$ as,
\begin{equation}
\frac{d(\delta^{\eta}z_{p})}{dt}+\textbf{a}_{p}(\delta^{\eta}z_{p})^{1+\textbf{c}_{p}}\leq
\textbf{b}_{p}
\label{E21}
\end{equation}
where \textbf{a}$_{p}$, \textbf{b}$_{p}$, \textbf{c}$_{p}$ and
\textbf{d}$_{3/2}\geq0$ are positive functions on $t$ and $b$ for
$p\in\left\{3/2,2,5/2,...\right\}$, and more importantly, they are
independent of the normalized moment
$\delta^{\eta}z_{p}$.
\end{Corollary}
\begin{proof}
The expressions for \textbf{a}$_{p}$ and \textbf{c}$_{p}$ can be
found by comparison between (\ref{E20}-\ref{E21}) and
(\ref{e16})
\begin{equation}\label{E21.1}
\textbf{a}_{p}(t)=(1-\gamma_{p})k_{\alpha}\Gamma(p+b)^{\alpha/2p}\delta^{\eta}m_{0}^{-\alpha/2p}
\:\mbox{and}\:\textbf{c}_{p}=\alpha/2p.
\end{equation}
Clearly they are positive functions of $t$ and independent of
$\delta^{\eta}z_{p}$.  For $p=3/2$ the short expression for
$\delta^{\eta}Z_{p}$ is obtained
\begin{align*}
\delta^{\eta}Z_{3/2}&=\max\left\{\delta^{\eta}(z_{1}z_{(1+\alpha)/2}),\delta^{\eta}
(z_{1+\alpha/2}z_{1/2})\right\}\\
&\leq\delta^{\eta}(z_{1}z_{(1+\alpha)/2})+\delta^{\eta}(z_{1+\alpha/2}z_{1/2})\\
&=\delta^{\eta}(z_{1}z_{(1+\alpha)/2})+\delta^{\eta^{-}}(z_{1+\alpha/2}z_{1/2})+
z_{1/2}\delta^{\eta}z_{1+\alpha/2} .
\end{align*}
But note that for $\alpha\in(0,1]$,
\begin{equation}\label{E21.10}
\delta^{\eta}z_{1+\alpha/2}\leq 1+\delta^{\eta}z_{3/2}.
\end{equation}
Therefore, this together with Lemma~\ref{L8} leads to (\ref{E20}) with
\begin{multline}\label{E21.11}
\textbf{d}_{3/2}(t)=\gamma_{3/2}k_{0}(3/2)^{\alpha/2+b}z_{1/2}(t)\:\:\mbox{and}\\
\textbf{b}_{3/2}(t)=\gamma_{3/2}k_{0}(3/2)^{\alpha/2+b}\left\{\delta^{\eta}
(z_{1}z_{(1+\alpha)/2})+\delta^{\eta^{-}}(z_{1+\alpha/2}z_{1/2})+z_{1/2}\right\}\\
+k_{1}(3/2)^{\alpha/2}\delta^{\eta^{-}}(m_{0}z_{(3+\alpha)/2})+\delta^{\eta^{-}}
(m_{\alpha/2}z_{3/2}).
\end{multline}
For $p\in\left\{2,5/2,3,...\right\}$ it is clear that for $1\leq
k\leq k_{p}$ the subindexes $k$, $p-k+\alpha/2$, $k+\alpha/2$ and
$p-k$ used in the definition of $\delta^{\eta}Z_{p}$ are all
strictly less that $p$.  Hence the term $\delta^{\eta}Z_{p}$ do not
depend on $\delta^{\eta}z_{p}$.  Therefore the expression
(\ref{E21}) follows if we select \begin{equation}\label{E21.2}
\textbf{b}_{p}(t)=\gamma_{p} k_{0}p^{\alpha/2+b}\delta^{\eta}Z_{p}+
\left\{k_{1}p^{\alpha/2}\delta^{\eta^{-}}(m_{0}z_{p+\alpha/2})+\delta^{\eta^{-}}
(m_{\alpha/2}z_{p})\right\}.
\end{equation} Recall in equation (\ref{E21.2}) that the functions
$\delta^{\eta^{-}}(m_{0}m_{p+\alpha/2})$ and
$\delta^{\eta^{-}}(m_{\alpha/2}m_{p})$ depend on lower derivatives
moments.
\end{proof}
%
\section{Main Results}
Theorem~\ref{T1} states that if the initial moments of the
derivatives of any order are finite, they will continue finite
through time.  Moreover, these moments are controlled by the initial
datum in a specific way.  The result is an extension of Bobylev
work, Theorem~\ref{T4} (item 1), which assures this behavior for the
regular $p$-moments.
\begin {Theorem}\label{T1} Let $\eta$ any
multi-index and assume that $\delta^{\eta}m_{0}>0$, and
$\delta^{\nu}m_{0},\delta^{\nu}m_{1}$ uniformly bounded in $[0,T]$
for $T>0$ and all $\nu\leq\eta$.  Also assume that for some
constants $k>0$ and $q\geq1$ the initial renormalized moments of the
solution's derivatives satisfy the grow condition on $p$
\begin{displaymath}
\delta^{\nu}z_{p}(0)\leq kq^{p},\;\;\;\;p=3/2,2,5/2,...
\end{displaymath}
Then we have the following uniform bound for the renormalized moments on $t\in~[0,T]$
\begin{equation}
\delta^{\nu}z_{p}(t)\leq KQ^{p},\;\;\;\;p=3/2,2,5/2...
\label{e18}
\end{equation}
for all $\nu\leq\eta$, some $Q\geq q$ and a positive constant
$K=K(\eta,\left\|f\right\|_{L^{\infty}([0,T];W^{|\eta|,1}_{2})},k)$
depending on the multi-index $\eta$, the constant $k$, and on the
$L^{\infty}([0,T];W^{|\eta|,1}_{2})$ norm of $f$.
\end{Theorem}
\begin{proof}
Argue by induction on the multi-index order $|\eta|$.  The case $|\eta|=0$ follows from a
direct application of Bobylev work~\cite{Bobylev1} for the hard spheres case $(\alpha=1)$ or
Gamba-Panferov-Villani work~\cite{Gamba} for the general hard potential case $\alpha\in(0,1)$,
see Theorem~\ref{T4} (item 1) on Appendix A.\\
\\
Thus, the induction hypothesis $\textbf{(IH)}$ reads:  Assume that
Theorem~\ref{T1} is true for any multi-index $\nu$ with
$|\nu|<|\eta|$, therefore there exists $K_{1}>0$ and $Q\geq q$
depending on different parameters as stated above, such that for
$t\in[0,T]$ and $|\nu|<|\eta|$
\begin{displaymath}
\delta^{\nu}z_{p}(t)\leq K_{1}Q^{p}\;\;\;\mbox{for}\;\;\;p\geq 1
\end{displaymath}
The purpose of the rest of the proof is to prove that
\begin{displaymath}
\delta^{\eta}z_{p}(t)\leq KQ^{p}\;\;\;\mbox{for}\;\;\;p=1,3/2,2,5/2,...
\end{displaymath}
for some $K>0$.  Recall the parameters $\textbf{a}_{p}$ and $\textbf{b}_{p}$ with
$p\in{3/2,2,5/2,...}$ defined explicitly in the proof Corollary~\ref{C1}.  The idea
is to use induction on $p$ to show that the quotient $\textbf{a}_{p}/\textbf{b}_{p}$
is bounded by $KQ^{p}$ and conclude by using Lemma~\ref{L9}.\\
Define
\begin{equation}
\textbf{a}^{p}_{*}\equiv\inf_{t\in[0,T]}\textbf{a}_{p}(t)=\left\|\delta^{\eta} m_{0}
\right\|_{L^{\infty}[0,T]}^{-\alpha/2p}(1-\gamma_{p})\Gamma(p+b)^{\alpha/2p}.
\label{E23.1}
\end{equation}
Observe that due to the induction hypothesis (\textbf{IH})
\begin{multline*}
\delta^{\eta-}(m_{0}z_{p+\alpha/2})\leq2^{|\eta|}K_{1}\left\|f\right\|_{L^{\infty}
([0,T];W^{|\eta|,1}_{1})}Q^{p+\alpha/2},\;\;\mathrm{and}\\\delta^{\eta-}(m_{\alpha/2}z_{p})
\leq2^{|\eta|}K_{1}\left\|f\right\|_{L^{\infty}([0,T];W^{|\eta|,1}_{1})}Q^{p}.
\end{multline*}
where we have used the control of the moments $0$ and $\alpha/2$ of
the $\eta$ derivative of $f$ provided by the
$L^{\infty}([0,T];W^{|\eta|,1}_{1})$ norm of $f$
\begin{displaymath}
\sum_{i=1,\alpha/2}\max_{\nu\leq\eta}\left\{\left\|\delta^{\nu}m_{i}
\right\|_{L^{\infty}[0,T]}\right\}\leq\left\|f\right\|_{L^{\infty}([0,T];W^{|\eta|,1}_{1})}.
\end{displaymath}
Hence, for $p\in{2,5/2,3,...}$
\begin{equation}\label{E23.2}
\textbf{b}_{p}(t)\leq\gamma_{p}k_{0}p^{\alpha/2+b}\delta^{\eta}Z_{p}+2^{|\eta|}K_{1}
\left\|f\right\|_{L^{\infty}([0,T];W^{|\eta|,1}_{1})}Q^{p}\left\{k_{1}Q^{\alpha/2}
p^{\alpha/2}+1\right\}.
\end{equation}
Substitute (\ref{E23.1}) and (\ref{E23.2}) in (\ref{E21}) to conclude that
\begin{multline}\label{E24}
\frac{d(\delta^{\eta}z_{p})}{dt}+\textbf{a}^{p}_{*}(\delta^{\eta}z_{p})^{1+\alpha/2p}
\leq\gamma_{p}k_{0}p^{\alpha/2+b}\delta^{\eta}Z_{p}\\+2^{|\eta|}K_{1}\left\|f
\right\|_{L^{\infty}([0,T];W^{|\eta|,1}_{1})}Q^{p}\left\{k_{1}Q^{\alpha/2}p^{\alpha/2}+1\right\}.
\end{multline}
\medskip
Next, define the following sequences of $p$
\begin{displaymath}
A^{1}_{p}=\frac{k_{0}\gamma_{p}p^{\alpha/2+b}}{\textbf{a}^{p}_{*}}\;\;\;\;\mbox{and}\;\;
\;A^{2}_{p}=2^{|\eta|}K_{1}\left\|f\right\|_{L^{\infty}([0,T];W^{|\eta|,1}_{1})}
\frac{k_{1}Q^{\alpha/2}p^{\alpha/2}+1}{\textbf{a}^{p}_{*}},
\end{displaymath} in this way equation (\ref{E24}) can be written as
\begin{equation}\label{E24.1}
\frac{d(\delta^{\eta}z_{p})}{dt}+\textbf{a}^{p}_{*}(\delta^{\eta}z_{p})^{1+\alpha/2p}\leq
\textbf{a}^{p}_{*}A^{1}_{p}\delta^{\eta}Z_{p}+\textbf{a}^{p}_{*}A^{2}_{p}Q^{p}.
\end{equation}
Let us, for the moment, divert our attention from equation
(\ref{E24.1}) and make an observation regarding the sequences
$\{A^{1}_{p}\}$ and $\{A^{2}_{p}\}$.  Recall the asymptotic formula
for $\Gamma(p+b)$ with $b>0$
 \begin{displaymath}
\Gamma(p+b)^{\alpha/2p}\sim p^{\alpha/2}\;\;\mbox{for large}\;p,
\end{displaymath}
also recall that in the prove of Lemma (\ref{L4})
\begin{displaymath}
\gamma_{p}\sim p^{-\epsilon/2}\;\;\mbox{for large}\;p.
\end{displaymath}
Therefore, by letting $b-\epsilon/2<0$ we have,
\begin{displaymath}
A^{1}_{p}~\sim p^{b-\epsilon/2}\rightarrow0\;\;\;\mbox{as}\;\;p\rightarrow \infty.
\end{displaymath}
Meanwhile, the sequence $\{A^{2}_{p}\}$ is bounded.  Define,
\begin{displaymath}
\left|A^{2}_{p}\right|_{\infty}\equiv\sup_{p\geq2}A^{2}_{p}<\infty.
\end{displaymath}
Thus, there exists $p_{0}$ such that
\begin{displaymath}
2^{|\eta|}K_{1}Q^{\alpha/2}A^{1}_{p}\leq 1/2\;\;\;\;\mathrm{if}\;\;\mathnormal p\geq p_{0}.
\end{displaymath}
Now, it is claimed that it is possible to take a number $K\geq$
max$\left\{1,k,K_{1},2\left|A^{2}_{p}\right|_{\infty}\right\}$ such
that
\begin{equation}\label{E24.2} \delta^{\eta} z_{p}(t)\leq
KQ^{p}\;\;\;\mbox{for}\;\;p=3/2,2,5/2...\;\;\mbox{and}\;\; p\leq
p_{0}.
\end{equation}
Let us prove that this $K$ actually exists by arguing as follows:  When $p=3/2$
\begin{displaymath}
\delta^{\eta}Z_{p}=\max\left\{\delta^{\eta}(z_{1}z_{(1+\alpha)/2}),\delta^{\eta}
(z_{1+\alpha/2}z_{1/2})\right\}.
\end{displaymath}
In the one hand, by hypothesis
\begin{displaymath}
\delta^{\eta}(z_{1}z_{(1+\alpha)/2})\leq\frac{2^{|\eta|}}{\Gamma((1+\alpha)/2+b)}
\left\|f\right\|^{2}_{L^{\infty}([0,T];W^{|\eta|,1}_{2})}<+\infty.
\end{displaymath}
and, in the other hand, by $\textbf{(IH)}$
\begin{align*}
\delta^{\eta}(z_{1+\alpha/2}z_{1/2})&\leq\frac{2^{|\eta|}K_{1}Q^{1+\alpha/2}}
{\Gamma(1/2+b)}(1+\delta^{\eta}z_{1+\alpha/2})\left\|f\right\|_{L^{\infty}([0,T];
W^{|\eta|,1}_{1})}\\
&\leq\frac{2^{|\eta|}K_{1}Q^{1+\alpha/2}}{\Gamma(1/2+b)}(2+\delta^{\eta}z_{3/2})
\left\|f\right\|_{L^{\infty}([0,T];W^{|\eta|,1}_{1})}.
\end{align*}
Combine these bounds with the definitions for $\textbf{b}_{3/2}$ and
$\textbf{d}_{3/2}$ in Corollary~\ref{C1} and apply Lemma~\ref{L9} to
find that $\delta^{\eta}z_{3/2}$ is bounded.  But once
$\delta^{\eta}z_{3/2}$ is bounded, $\delta^{\eta}Z_{2}$ is
immediately bounded and hence, by a new application of
Lemma~\ref{L9} on (\ref{E24.1}) for $p=2$, $\delta^{\eta}z_{2}$ is
bounded.  Repeat this process up to $p_{0}$ to find that
$\delta^{\eta}z_{p}(t)$ is bounded for $p=3/2,2,5/2,...,p_{0}$.  So
it is just a matter of choosing $K>0$ sufficiently large so that
(\ref{E24.2}) is fulfill.\\
\\
Let us argue by induction on the integrability index $p$ to show
that the same constants $K$ and $Q$ also hold for $p>p_{0}$ with
$p\in\left\{3/2,2,...\right\}$.  Assume that (\ref{E24.2}) holds.
Hence, observing that the term $\delta^{\eta}Z_{p}$ does not depend
on $\delta^{\eta}z_{p}$ for $p>3/2$ and using $\textbf{(IH)}$ one
concludes that
\begin{displaymath}
\delta^{\eta}Z_{p}<2^{|\eta|}K_{1}KQ^{p+\alpha/2}.
\end{displaymath}
Therefore, inequality (\ref{E24.1}) reads for $p>p_{0}$
\begin{equation}
\frac{d(\delta^{\eta}z_{p})}{dt}+\textbf{a}^{p}_{*}(\delta^{\eta}z_{p})^{1+\alpha/2p}
\leq 2^{|\eta|}\textbf{a}^{p}_{*}A^{1}_{p}K_{1}KQ^{p+\alpha/2}+\textbf{a}^{p}_{*}
A^{2}_{p}Q^{p}=\textbf{b}^{p}_{*}
\end{equation}
thus by Lemma~\ref{L9}
\begin{displaymath}
\delta^{\eta}z_{p}(t)\leq \max\left\{(\textbf{b}^{p}_{*}/\textbf{a}^{p}_{*})^{2p/(2p+\alpha)},
\delta^{\eta} z_{p}(0)\right\}.
\end{displaymath}
But the condition $p>p_{0}$ and the choice of $K$ implies that
\begin{displaymath}
\textbf{b}_{p}(t)/\textbf{a}_{p}(t)\leq\textbf{b}^{p}_{*}/\textbf{a}^{p}_{*}=
\left\{2^{|\eta|}A^{1}_{p}K_{1}Q^{\alpha/2}+\frac{A^{2}_{p}}{K}\right\}KQ^{p}
\leq KQ^{p}\;\;\;\mbox{for}\;\;\;p>p_{0}.
\end{displaymath}
Since same inequality holds for $p\leq p_{0}$ one concludes that,
\begin{displaymath}
\delta^{\eta}z_{p}\leq \max\left\{KQ^{p},kq^{p}\right\}= KQ^{p}\;\;\;\mbox{for}
\;\;p=1,3/2,2,...
\end{displaymath}
This completes the proof.
\end{proof}
$Remarks:$
\begin{itemize}
\item For any $p>1$ a simple Lebesgue interpolation argument together with
Theorem~\ref{T1} shows that $\delta^{\eta}z_{p}\leq K_{\eta}Q^{p}$.
\item The growing constant $q$ for the initial datum is in general smaller that
the one obtained for the differential moments.  Thus, the control on the differential
moments may worsen depending on the initial conditions.
\item The following is a different way to state Theorem~\ref{T1}: Let $\eta$ any
multi-index and assume that $f_{0}\in L^{1}_{2}$ and $f\in L^{\infty}([0,T];W^{|\eta|,1}_{2})$,
 if for some $r_{0}>0$ we have that $\int_{\mathbb{R}^{n}}|\partial^{\nu}f_{0}
 |\exp(r_{0}|\xi|^{2})d\xi<\infty$ for all $\nu\leq\eta$ then \[\sup_{[0,T]}
 \left\{\int_{\mathbb{R}^{n}}|\partial^{\nu}f|\exp(r|\xi|^{2})d\xi\right\}<\infty
\]
for some $r\leq r_{0}$ and all $\nu\leq\eta$.
\end{itemize}
Lemma~\ref{L10} and Lemma~\ref{L11} prove that, for a solution of
Boltzmann equation $f$, the differential moments $\delta^{\nu}m_{0}$
and $\delta^{\nu}m_{1}$ are uniformly bounded on time for
$\nu\leq\eta$ , provided we have sufficient regularity in the
initial datum $f_{0}$.  In other words, given sufficient regularity
on $f_{0}$ we should have that $f\in
L^{\infty}([0,T];W^{|\eta|,1}_{2})$.
\begin{Lemma}\label{L10} Let
$\eta$ any multi-index and suppose that $f_{0}\in
W^{|\eta|,1}_{2+\alpha}$, then for any $T\in(0,\infty)$ we have that
$f\in L^{\infty}([0,T];W^{|\eta|,1}_{2+\alpha})$.  Moreover,
$\delta^{\eta}m_{0}(t)>0$ always holds.
\end{Lemma}
\begin{proof}
First note that if $\delta^{\eta}m_{0}(t')=0$ for some fixed $t'>0$,
we have that $\partial^{\eta}f(\xi,t')=0$.  Therefore $f(\xi,t')$
would be a polynomial in the variables $\xi_{i}$ with
$i=1,2,\cdots,n$.  Hence $f(\xi,t')$ would not be integrable unless
$f(\xi,t')=0$.  But
$0=\left\|f(\cdot,t')\right\|_{L^{1}}=\left\|f_{0}\right\|_{L^{1}}$
due to mass conservation.  This is impossible for a non zero initial
datum.\\
\\
Next, since $A[1]=0$ and $A[|\xi|^{2}]\leq0$, one uses
Lemma~\ref{L3} to obtain the following inequalities
\begin{multline*}
1/2\frac{d(\delta^{\eta}m_{0})}{dt}\leq \delta^{\eta}m_{0}m_{\alpha/2}+
\delta^{\eta}m_{\alpha/2}m_{0}\\+\sum_{0<\nu<\eta}\left(\begin{array}{c}\eta\\
\nu\end{array}\right)\left(\delta^{\nu}m_{\alpha/2}\delta^{\eta-\nu}m_{0}+
\delta^{\nu}m_{0}\delta^{\eta-\nu}m_{\alpha/2}\right)
\end{multline*}
and
\begin{multline}
1/2\frac{d(\delta^{\eta}m_{1})}{dt}\leq
\delta^{\eta}m_{0}m_{1+\alpha/2}+\delta^{\eta}m_{\alpha/2}m_{1}\\+\sum_{0<\nu<\eta}
\left(\begin{array}{c}\eta\\\nu\end{array}\right)
\left(\delta^{\nu}m_{\alpha/2}\delta^{\eta-\nu}m_{1}+\delta^{\nu}m_{0}
\delta^{\eta-\nu}m_{1+\alpha/2}\right).
\label{D1}
\end{multline}

\noindent We can now conclude the proof by using inequality
(\ref{D1}) in order to implement an induction argument on the index
order $|\eta|$. Note that for the case $|\eta|=0$, the conservation
of mass and dissipation of energy implies that $f\in
L^{\infty}([0,T];L^{1}_{2})$.  In addition, since $f_{0}\in
L^{1}_{2+\alpha}$, the moment $1+\alpha/2$ is finite in the initial
datum, then we must have that this moment is uniformly bounded in
time, for this is precisely the work of
Gamba-Panferov-Villani~\cite{Gamba}.  Hence, $f\in
L^{\infty}([0,T];L^{1}_{2+\alpha})$.\\
\\
For $|\eta|>0$, take $f_{0}\in W^{|\eta|,1}_{2+\alpha}$ and assume
that the result is valid for all $|\nu|<|\eta|$.  Since
$W^{|\eta|,1}_{2+\alpha}\subset W^{|\nu|,1}_{2+\alpha}$ then
$f_{0}\in W^{|\nu|,1}_{2+\alpha}$, thus by induction hypothesis we
have that $f\in L^{\infty}([0,T];W^{|\nu|,1}_{2+\alpha})$ for all
$|\nu|<|\eta|$. Therefore, $\delta^{\nu}m_{0}$, $\delta^{\nu}m_{1}$
and $\delta^{\nu}m_{1+\alpha/2}$ are uniformly bounded on $[0,T]$ as
long as $|\nu|<|\eta|$.  Note, that
\begin{displaymath}
\delta^{\eta}m_{\alpha/2}\leq\delta^{\eta}m_{0}+\delta^{\eta}m_{1}.
\end{displaymath}
As a result, inequalities (\ref{D1}) imply that $\delta^{\eta}m_{0}$
and $\delta^{\eta}m_{1}$ are uniformly bounded on $[0,T]$, i.e.
$f\in L^{\infty}([0,T];W^{|\eta|,1}_{2})$.  But
$\delta^{\eta}m_{1+\alpha/2}(0)$ is finite by hypothesis, thus we
can apply Theorem~\ref{T1} again to get that
$\delta^{\eta}m_{1+\alpha/2}(t)$ is finite in $[0,T]$.  We conclude
that $f\in L^{\infty}([0,T];W^{|\eta|,1}_{2+\alpha})$.
\end{proof}
Lemma~\ref{L11} shows that it is possible to go further and obtain a
global in time result for the elastic case, provided that more
regularity on $f_{0}$ is imposed.
\begin{Lemma}\label{L11}
Let $\eta$ any multi-index and assume that $f_{0}\in
W^{|\eta|,1}_{2+\alpha}\cap H^{|\eta|}_{(|\eta|-1)(1+\alpha/2)}$
then $f\in L^{\infty}(\mathbb{R}^{+};W^{|\eta|,1}_{2+\alpha})$.
\end{Lemma}
\begin{proof}
In the one hand, for all multi-index $\nu$ satisfying $\nu\leq\eta$
we have by Cauchy-Schwartz inequality that
\begin{displaymath}
\delta^{\nu}m_{p}\leq C_{s,n}\left\|f\right\|_{H^{|\eta|}_{2p+s/2}}
\end{displaymath}
for any $s>n$ and some constant $C_{s,n}$ depending on $s$ and the dimension $n$.
Therefore, by letting $p=1+\alpha/2$ we obtain,
\begin{displaymath}
\max_{\nu\leq\eta}\left\{\delta^{\nu}m_{0}(t),\delta^{\nu}m_{1}(t),
\delta^{\nu}m_{1+\alpha/2}(t)\right\}\leq
C_{s,n}\left\|f(t,\cdot)\right\|_{H^{|\eta|}_{2+\alpha+s/2}}.
\end{displaymath} Then, using Theorem (\ref{T6}) in Appendix A
\begin{displaymath}
\sup_{\;t\geq
t_{0}}\left\{\max_{\nu\leq\eta}\left\{\delta^{\nu}m_{0}(t),\delta^{\nu}m_{1}(t),
\delta^{\nu}m_{1+\alpha/2}(t)\right\}\right\}<+\infty.
\end{displaymath}
On the other hand, the differential moments are bounded for $t\leq
t_{0}$ by Lemma~\ref{L10} under these assumptions on $f_{0}$. Hence,
they are bounded uniformly for all $t>0$.  As a result, $f\in
L^{\infty}(\mathbb{R}^{+};W^{|\eta|,1}_{2+\alpha})$.
\end{proof}
The results of Theorem~\ref{T1}, Lemma~\ref{L10} and Lema~\ref{L11}
can be readily used to obtain the $L^{1}$-Maxwellian bound for
derivatives of any order.
\begin{Theorem}\label{T2} Let $\eta$ any
multi-index and assume that $f_{0}\in W^{|\eta|,1}_{2+\alpha}$.  In
addition, assume the grow condition on the initial moments
\begin{displaymath}
\delta^{\nu}m_{p}(0)/p!\leq k\;q^{p}
\end{displaymath}
for $p\geq 3/2$, all $\nu\leq\eta$ and some positive constants $k$
and $q$.  Then, $\partial^{\nu}f$ has exponential tail of order 2 in
$[0,T]$ for $\nu\leq\eta$ and $T\in(0,T)$.  Moreover, if we
additionally assume $f_{0}\in H^{|\eta|}_{(|\eta|-1)(1+\alpha/2)}$
then the conclusion can be extended to $T=+\infty$.
\end{Theorem}
\begin{proof}
Fix $T\in(0,\infty)$ and observe that using Lemma~\ref{L10} it is
possible to conclude that $f\in
L^{\infty}([0,T];W^{|\eta|,1}_{2+\alpha})$.  From this follows that
the moments $\delta^{\nu}m_{0}$ and $\delta^{\nu}m_{1}$ are bounded
in $[0,T]$ for all $\nu\leq\eta$.  Therefore, the conditions of
Theorem~\ref{T1} are fulfilled and we can use it to conclude that
for all $\nu\leq\eta$ the following inequality holds in $[0,T]$
\begin{equation}\label{D2}
\int_{\mathbb{R}^{n}}|\partial^{\nu}f|e^{r|v|^{2}}dv=\sum_{i}
\frac{\delta^{\nu} m_{i}}{i!}\;r^{i}\leq
K\sum_{i}\frac{\Gamma(i+b)}{\Gamma(i+1)}\;\left(Qr\right)^{i},
\end{equation}
where $Q\geq q$ and $K>0$ are constants that depend on different
parameters as discussed in Theorem~\ref{T1}. But,
\begin{displaymath} \frac{\Gamma(i+b)}{\Gamma(i+1)}\sim
i^{b-1}\;\;\mbox{for large}\;\;i.
\end{displaymath}
Consequently, the sum behave like
\begin{displaymath}
\sum_{i}i^{b-1}(Qr)^{i}
\end{displaymath}
Thus, it suffices to choose $r>0$ such that $Qr<1$ so that the sum in (\ref{D2}) converges.\\
\\
Use the assumption that $f_{0}\in H^{|\eta|}_{(|\eta|-1)(1+\alpha/2)}$ and apply
Lemma~\ref{L11} to extend the result to the limit case $T=+\infty$.
\end{proof}
$Remark:$
\begin{itemize}
\item As a final remark on Theorem~\ref{T1}, Lemma~\ref{L10} and Lemma~\ref{L11},
observe that for any multi-index $\eta$ and $k\geq 2+\alpha$, Theorem~\ref{T1}
implies that if $f_{0}\in W^{|\eta|,1}_{k}$, then $f\in C([0,T];W^{|\eta|,1}_{k})$
for any $T<\infty$.  For the elastic case $T=\infty$ is also allowed, provided we
have that $f_{0}\in H^{|\eta|}_{(|\eta|-1)(1+\alpha/2)}$.
\end{itemize}
\section {Proof of Theorem \ref{T3}}
In order to simplify the notation set $Q^{-}(f,g)=f\cdot L(g)$ where
\begin{displaymath}
L(g)=\int_{\mathbb{R}^{n}}g_{*}|\xi-\xi_{*}|^{\alpha}d\xi_{*}.
\end{displaymath}
\begin{proof}
Differentiate the equation (\ref{e1}) $\eta$ times in velocity and
multiply the result by $\mathrm{sgn}\mathnormal(\partial^{\eta}f)$
to obtain
\begin{multline}\label{e20}
\partial_{t}(|\partial^{\eta}f|)+|\partial^{\eta}f|\; L(f)\leq Q^{+}(|\partial^{\eta}
f|,f)+Q^{+}(f,|\partial^{\eta}f|)+f\cdot L(|\partial^{\eta}f|)\\+\sum_{0<\nu<\eta}
\left(\begin{array}{c}\eta\\\nu\end{array}\right)
\left\{Q^{+}(|\partial^{\nu}f|,|\partial^{\eta-\nu}f|)+Q^{-}(|\partial^{\nu}f|,
|\partial^{\eta-\nu}f|)\right\}. \end{multline}
We use equation (\ref{e20}) to argue by induction on the index order $|\eta|$.
The case $|\eta|=0$ follows directly from Theorem~\ref{T4}, item (2).\\
Next, let $f_{0}$ fulfilling all the conditions of the Theorem and
assume the result for $|\nu|<|\eta|$.  Then, there exists $r'\leq
r_{0}$ such that for any $|\nu|<|\eta|$
\begin{displaymath}
|\partial^{\nu}f|\leq
K^{1}_{\eta,r_{0}}(1+|\xi|^{2})^{|\nu|/2}M_{r'}
\end{displaymath}
where $K^{1}_{\eta,r_{0}}$ is a positive constant depending on $\eta$ and $r_{0}$.\\
By hypothesis, $|\partial^{\nu}f_{0}|/M_{r_{0}}\in L^{1}$ for all
$\nu\leq\eta$.  Thus, the grow condition required in
Theorem~\ref{T2} on the derivative moments of the initial datum
$f_{0}$ is satisfied, namely, that for some positive constants $k$
and $q$,
\begin{displaymath}
\delta^{\nu}m_{p}(0)/p!\leq kq^{p}\;\;\mbox{for}\;\;p\geq 0.
\end{displaymath}
Furthermore, $f_{0}\in H^{|\eta|}_{(|\eta|-1)(1+\alpha/2)}$, as a
result, Theorem~\ref{T2} applies to obtain that for some $r''\leq
r_{0}$
\begin{equation}
\sup_{\;t\geq0}\int_{\mathbb{R}^{n}}|\partial^{\nu}f|\;\exp(r''|\xi|^{2})d\xi=
\sup_{\;t\geq0}\left\|\partial^{\nu}f/M_{r''}\right\|_{L^{1}}<\infty
\label{E28}
\end{equation}
for all $\nu\leq\eta$.  Indeed, recall that in Theorem~\ref{T1} a
bigger grow constant $Q\geq q$ was obtained for controlling the
derivative's moments through time. Hence, previous integral must
converge in general for $r''\leq r_{0}$.\\
\\
Let $r=\min\left\{r',r''\right\}$ and divide inequality (\ref{e20})
by $M_{r}$. Using the induction hypothesis we can bound the
derivatives of lower order in (\ref{e20}) to get the inequality,
\begin{multline*}
\partial_{t}(|\partial^{\eta}f|/M_{r})+|\partial^{\eta}f/M_{r}|\; L(f)\leq\\
\frac{K^{1}_{\eta,r_{0}}}{M_{r}} \left\{Q^{+}(|\partial^{\eta}f|,M_{r})+Q^{+}
(M_{r},|\partial^{\eta}f|)\right\}+K^{1}_{\eta,r_{0}}\;L(|\partial^{\eta}f|)+\\
\frac{K^{1}_{\eta,r_{0}}}{M_{r}}\sum_{0<\nu<\eta}\left(\begin{array}{c}\eta\\
\nu\end{array}\right)Q^{+}((1+|\xi|^{2})^{|\nu|/2}M_{r},|\partial^{\eta-\nu}f|)+
Q^{-}((1+|\xi|^{2})^{|\nu|/2}M_{r},|\partial^{\eta-\nu}f|).
\end{multline*}
Use Theorem~\ref{T5} and Theorem~\ref{T6} in the Appendix A, to
obtain the following $L^{1}$ control from the previous
inequality
\begin{multline}\label{e21}
\partial_{t}(|\partial^{\eta}f|/M_{r})+|\partial^{\eta}f/M_{r}|\; L(f)\leq\\
K^{2}_{\eta,r_{0}}(1+|\xi|^{2})^{(|\eta|-1)/2}\sum_{0<\nu\leq\eta}\left(\begin{array}{c}\eta\\
\nu\end{array}\right)\left\|\partial^{\nu}f/M_{r}\right\|_{L^{1}}+L(|\partial^{\nu}f|)
\end{multline}
where $K^{2}_{\eta,r_{0}}>0$ is a constant depending on $\eta$,
$r_{0}$ and on the kernel $b(\cdot)$, as Theorem~\ref{T5}
states.\\
However, observe that for all $\nu$
\begin{align*}
L(|\partial^{\nu}f|)\leq |\xi|^{\alpha}\delta^{\nu}m_{0}+\delta^{\nu}m_{\alpha/2}&\leq
\mathrm{Const.}\;\left\|\partial^{\nu}f/M_{r}\right\|_{L^{1}}(1+|\xi|^{2})^{\alpha/2}\\
&\leq \mathrm{Const.}\;\left\|\partial^{\nu}f/M_{r}\right\|_{L^{1}}(1+|\xi|^{2})^{1/2}.
\end{align*}
Therefore, combining this inequality with by (\ref{E28}) we conclude
that the right-hand side of (\ref{e21}) is bounded by
$K^{3}_{\eta,r_{0}}(1+|\xi|^{2})^{|\eta|/2}$.
Specifically,
\begin{equation}
\partial_{t}(|\partial^{\eta}f|/M_{r})+|\partial^{\eta}f/M_{r}|\; L(f)\leq
 K^{3}_{\eta,r_{0}}(1+|\xi|^{2})^{|\eta|/2}.
\label{E30}
\end{equation}
Fix $t_{0}>0$ and integrate (\ref{E30}) over $[0,t_{0}]$.  It
follows that for any $t\in[0,t_{0}]$
\begin{displaymath}
|\partial^{\eta}f|/M_{r}\leq K^{3}_{\eta,r_{0}}\;
t_{0}(1+|\xi|^{2})^{|\eta|/2}+|\partial^{\eta}f_{0}|/M_{r}\leq
K^{4}_{\eta,r_{0}}\; t_{0}(1+|\xi|^{2})^{|\eta|/2}
\end{displaymath}
where $K^{4}_{\eta,r_{0}}$ is a positive constant that depends on
$\eta$, $r_{0}$ and the kernel $h(\cdot)$.\\
For $t>t_{0}$ use the
lower bound that provides Theorem~\ref{T4} (item 3) in the Appendix
to conclude that $C\equiv\inf_{\;\xi,t\geq t_{0}}L(f)>0$, thus using
the full differential inequality (\ref{E30})
\begin{displaymath}
|\partial^{\eta} f|/M_{r}\leq \max\left\{C^{-1}K^{3}_{\eta,r_{0}}(1+|\xi|^{2})^{|\eta|/2},
|\partial^{\eta}f_{0}|/M_{r}\right\}\leq K^{5}_{\eta,r_{0}}(1+|\xi|^{2})^{|\eta|/2}.
\end{displaymath}
Therefore, $K_{\eta,r_{0}}=$max$\left\{K^{4}_{\eta,r_{0}}\cdot
t_{0},K^{5}_{\eta,r_{0}}\right\}$ provides a sufficiently large
constant for any $t\geq0$.   Since it is possible to fix any time
$t_{0}$ to perform these calculations, this constant just depends on
$\eta$, $r_{0}$ and the kernel $h(\cdot)$.
\end{proof} $Remarks:$
\begin{itemize}
\item If assumption $f_{0}\in H^{|\eta|}_{(|\eta|-1)(1+\alpha/2)}$ is not imposed,
Theorem~\ref{T3} is still valid changing in the conclusion $"\sup_{\;t\geq0}"$ for
$"\sup_{\;0\leq t\leq T}"$ with $T$ finite.  This is a direct consequence of the
fact that Theorem~\ref{T2} is valid under these conditions for any finite time $T$.
\item Take as hypothesis of Theorem~\ref{T3} only that $f_{0}\in H^{|\eta|}_{(|\eta|-1)
(1+\alpha/2)}$ and
\[|\partial^{\nu}f_{0}|/\left\{(1+|\xi|^{2})^{|\nu|/2}M_{r_{0}}\right\}\in L^{\infty}\]
for $\nu\leq\eta$ and some positive $r_{0}$.  Since for any $r'\in(0,r_{0})$ the last
hypothesis implies that $|\partial^{\nu}f_{0}|/M_{r'}\in L^{1}$ and
\[|\partial^{\nu}f_{0}|/\left\{(1+|\xi|^{2})^{|\nu|/2}M_{r'}\right\}\in L^{\infty}\]
for all $\nu\leq\eta$.  Thus, using Theorem~\ref{T3}, there exist $r\leq r'<r_{0}$ such that
\begin{displaymath}
\sup_{\;t\geq 0}\frac{|\partial^{\nu}f|}{(1+|\xi|^{2})^{|\nu|/2}M_{r}}\leq K_{\eta,r'}.
\end{displaymath}
\item Mischler et al. \cite{Mischler1} proved that for inelastic collisions
the solution of the problem (\ref{e1}) converges to the Dirac delta distribution
as the time goes to infinity (see \cite{Mischler1}).  This is a consequence of the
energy loss and therefore the cooling process that is taking place in the gas.
Thus, for this case, it is not possible to obtain results like Theorem~\ref{T3}
which involve bounds that are uniformly in time for the solution $f$.  In the
elastic case, the gas does not have this cool down phenomena hence uniform bounds
on the derivatives can be proved in $[0,\infty)$.
\end{itemize}
%
\appendix
\section{Facts for a solution $f$ of the Homogeneous Boltzmann Problem}
The homogeneous Boltzmann problem for hard and Maxwellian potentials is nowadays
pretty well understood, in addition to existence and uniqueness of solutions
\cite{Mischler} many other results are available like positive estimates
\cite{Pulvirenti} and propagation of regularity \cite{Mouhot}.
The most useful results used in this work are stated by the following theorems.
\begin {Theorem}\label{T4}
Assume that $f_{0}$ and $h(\cdot)$ have the properties discussed in
the introduction and that $\alpha\in (0,1]$. Then the following
properties holds for a solution $f$ of the elastic homogeneous
Boltzmann problem:
\begin{description}
\item [\it(1)]If $f_{0}$ satisfies $\int_{\mathbb{R}^{n}}f_{0}\exp(r_{0}|\xi|^{2})
d\xi<\infty$ for some $r_{0}>0$, then there exist $r\leq r_{0}$ such
that
$\sup_{\;t\geq0}\int_{\mathbb{R}^{n}}f\exp(r|\xi|^{2})d\xi<\infty$.
\item
[\it(2)]If $f_{0}\leq K_{0}\exp(-r_{0}|\xi|^{2})$ for some
$K_{0},r_{0}>0$ then there exist $r\leq r_{0}$ such that $f\leq
K\exp(-r|\xi|^{2})$ for all $t\geq0$ and some positive constants
$K$.
\item [\it(3)]For every $t_{0}>0$ there are positive constants
$K,r_{0}$ such that $f(t,\xi)\geq K \exp(-r_{0}|\xi|^{2})$ for all
$t\geq t_{0}$.
\end{description}
\end{Theorem}
These are precisely the results that we want to extend for the
derivative of $f$ and their proof can be found in \cite{Bobylev} for
item $1$, also \cite{Gamba} for item $2$ and \cite{Pulvirenti} for
item $3$.  Of course item $3$ is not true in general for
$|\partial^{\eta}f|$, for example as shown by a Maxwellian solution,
the gradient can be in general zero in some points of the velocity
space at a given time. However, this result will prove to be helpful
in showing pointwise bounds for the derivatives of a solution.
Observe also that in items $1$ and $2$ in Theorem~\ref{T4} the rate
of decay $r_{0}$ that controls $f_{0}$ is worsen in general to
$r\leq r_{0}$ for controlling $f$.\\
Next, we state a remarkable
result essential to prove item $2$ in the previous Theorem, in
particular, essential to control the gain collision operator.
\begin{Theorem}\label{T5}
Assume $B(u,\sigma)=|u|^{\alpha}h(\hat{u}\cdot\sigma)$ with
$h(\cdot)$ satisfying the conditions stated in the introduction.
Then for any measurable function $g\geq0$,
\begin{displaymath}
\left\|\frac{Q^{+}(g,M_{r})}{M_{r}}\right\|_{L^{\infty}}\leq
K\left\|\frac{g}{M_{r}}\right\|_{L^{1}}
\end{displaymath}
for some positive constant $K$ depending on $\alpha$ and $r$.
\end{Theorem}
As usual in the $L^{\infty}$ bounds for $Q^{+}(f,f)$, this result is
a direct application of the Carleman representation formula and
clever manipulations of it.  This Theorem is very helpful when we
try to prove an $L^{\infty}$ bound for the derivatives of $f$.  The
proof of Theorem~\ref{T5} can be found on \cite[Lemma
12]{Gamba}.\\
It is clear that same result holds for
$Q^{+}(M_{r},g)$, moreover, and slightly modification of the proof
can be used to obtain the following Theorem.
\begin{Theorem}\label{T6}
Assume $B(u,\sigma)=|u|^{\alpha}h(\hat{u}\cdot\sigma)$ with
$h(\cdot)$ satisfying the conditions stated in the introduction,
then for any measurable function $g\geq0$
\begin{displaymath}
\left\|\frac{Q^{+}(g,(1+|\xi|^{2})^{s}M_{r})}{(1+|\xi|^{2})^{s}M_{r}}\right\|_{L^{\infty}}\leq
K\left\|\frac{g}{M_{r}}\right\|_{L^{1}},
\end{displaymath}
for any $s>0$ and some positive constant $K$ depending on $s$, $\alpha$ and $\beta$.
\end{Theorem}
Finally, a powerful result proved by Mouhot and Villani
\cite[Theorem 4.2]{Mouhot} is also used.  This result helps to
obtain uniform bounds for infinite time for the derivative's
moments.  A small piece of this theorem, which is the one of use for
us, is stated below.
\begin{Theorem}\label{T7} Let $\alpha\in(0,2)$,
$s\in\mathbb{N}$ and assume that $f_{0}\in L^{1}_{2}\cap
H^{s}_{(s-1)(1+\alpha/2)}$.  Then for any $t_{0}>0$ and $k>0$,
\begin{displaymath}
\sup_{\;t\geq t_{0}}\left\|f(t,\cdot)\right\|_{H^{s}_{k}}<+\infty.
\end{displaymath}
This quantity depends on an upper bound on $L^{1}_{2}$ and
$H^{s}_{(s-1)(1+\gamma/2)}$ norms of $f_{0}$ and a lower bound on
$t_{0}$.
\end{Theorem}
The proof of this Theorem is rather technical
and requires several previous results on the control of the positive
collision operator including the gain of regularity of the positive
operator, however its spirit is, as in this work, to find a stable
differential equation for the $H^{s}$ norm of $f$ and proceed by
induction.

\Ack{This research author was partially supported by NSF under grant
DMS-0507038.  Support from the Institute from Computational
Engineering and Sciences at the University of Texas at Austin is
also gratefully acknowledged.}

%

\signra
\signig
%
\end{document}